\documentclass{article}
\usepackage[numbers,sort]{natbib}

\usepackage{amsmath}
\usepackage{amssymb}
\usepackage{amsthm}
\usepackage{tikz}
\usetikzlibrary{shapes.geometric}
\usepackage{caption}
\usepackage[hidelinks]{hyperref}

\DeclareMathOperator{\fs}{fs}
\DeclareMathOperator{\elm}{elm}

\newtheorem{theorem}{Theorem}
\newtheorem{claim}{Claim}[theorem]
\newtheorem{proposition}[theorem]{Proposition}
\newtheorem{lemma}[theorem]{Lemma}
\newtheorem{corollary}[theorem]{Corollary}

\newcount\savedtheorem

\def\claimqed{\smash{\scalebox{.75}[0.75]{$(\square$)}}}

\title{Edge Contraction and Forbidden Induced Subgraphs}
\author{Hany Ibrahim \qquad  Peter Tittmann\\University of Applied Sciences Mittweida}
\date{}

\begin{document}

\maketitle

\begin{abstract}
	Given a family of graphs $\mathcal{H}$, a graph $G$ is $\mathcal{H}$-free if any subset of $V(G)$ does not induce a subgraph of $G$ that is isomorphic to any graph in $\mathcal{H}$. We present sufficient and necessary conditions for a graph $G$ such that $G/e$ is $\mathcal{H}$-free for any edge $e$ in $E(G)$. Thereafter, we use these conditions to characterize claw-free, $2K_{2}$-free, $P_{4}$-free, $C_{4}$-free, $C_{5}$-free, split, pseudo-split, and threshold graphs.	
\end{abstract}

\section{Introduction}
A graph $G$ is an ordered pair $(V(G),E(G))$ where $V(G)$ is a set of vertices and $E(G)$ is a set of $2$-elements subsets of $V(G)$ called edges. The set of all graphs is $\mathcal{G}$. The degree of a vertex $v$, denoted by $deg(v)$, is the number of edges incident to $v$. We denote the maximum degree of a vertex in a graph $G$ by $\Delta(G)$. We call two vertices adjacent if there is an edge between them, otherwise, we call them nonadjacent. Moreover, the set of all vertices adjacent to a vertex $v$ is called the \emph{neighborhood} of $v$, which we denote by $N(v)$. On the other hand, the \emph{closed neighborhood} of $v$, denoted by $N[v]$, is $N(v) \cup \{v\}$. Generalizing this to a set of vertices $S$, the neighborhood of $S$, denoted by $N(S)$, is defined by $N(S) := \bigcup_{v\in S} N(v) - S$. Similarly the closed neighborhood of $S$, denoted by $N[S]$, is $N(S) \cup S$. Moreover, for a subset of vertices $S$, we denote the set of vertices in $S$ that are adjacent to $v$ by $N_{S}(v)$. Furthermore, we write $v$ is adjacent to $S$ to mean that $S \subseteq N(v)$ and $v$ is adjacent to exactly $S$ to mean that $S = N(v)$.

A set of vertices $S$ is \emph{independent} if there is no edge between any two vertices in $S$. We call a set $S$ \emph{dominating} if $N[S] = V(G)$. A subgraph $H$ of a graph $G$ is a graph where $V(H) \subseteq V(G)$ and $E(H) \subseteq E(G)$. An \emph{induced} graph $G[S]$ for a given set $S \subseteq V$, is a subgraph of $G$ with vertex set $S$ and two vertices in $G[S]$ are adjacent if and only if they are adjacent in $G$. Two graphs $G,H$ are \emph{isomorphic} if there is a bijective mapping $f: V(G) \to V(H)$ where $u,v \in V(G)$ are adjacent if and only if $f(u),f(v)$ are adjacent in $H$. In this case we call the mapping $f$ an isomorphism. Two graphs that are not isomorphic are called \emph{non-isomorphic}. In particular, an isomorphism from a graph to itself is called \emph{automorphism}. Furthermore, two vertices $u,v$ are similar in a graph $G$ if there is an automorphism that maps $u$ to $v$. The set of all automorphisms of a graph $G$ forms a group called the automorphism group of $G$, denoted by \emph{Aut($G$)}. The complement of a graph $G$, denoted by $\bar{G}$, is a graph with the same vertex set as $V(G)$ and two vertices in $\bar{G}$ are adjacent if and only if they are nonadjacent in $G$.

The \emph{independence number} of a graph $G$, denoted by $\alpha(G)$, is the largest cardinality of an independent set in $G$. In this thesis, we write singletons $\{x\}$ just as $x$ whenever the meaning is clear from the context. A vertex $u$ is a \emph{corner dominated} by $v$ if $N[u] \subseteq N[v]$.
Let $\mathcal{H}$ be a set of graphs. A graph $G$ is called \emph{$\mathcal{H}$-free} if there is no induced subgraph of $G$ that is isomorphic to any graph in $\mathcal{H}$, otherwise, we say $G$ is \emph{$\mathcal{H}$-exist}. 

By \emph{contracting} the edge between $u$ and $v$, we mean the graph constructed from $G$ by adding a vertex $w$ with edges from $w$ to the union of the neighborhoods of $u$ and $v$, followed by removing $u$ and $v$. We denote the graph obtained from contracting $uv$ by $G/uv$. If $e$ is the edge between $u$ and $v$, then we also denote the graph $G/uv$ by $G/e$. Further, we call $G/e$ a $G$-contraction. 
Finally, any graph in this paper is simple. For notions not defined, please consult \cite{bondy2000graph}. Additionally, we divide longer proofs into smaller claims, and we prove them only if their proofs are not apparent.

For a graph invariant $c$, a graph $G$, and a $G$-contraction $H$, the question of how $c(G)$ differs from $c(H)$ is investigated for different graph invariants. For instance, how contracting an edge in a graph affects its $k$-connectivity. Hence, the intensively investigated (\cite{kriesell2002survey}) notion of \emph{$k$-contractible} edges in a $k$-connected graph $G$ is defined as the edge whose contraction yields a $k$-connected graph. Another example is in the game Cops and Robber where a policeman and a robber are placed on two vertices of a graph in which they take turns to move to a neighboring vertex. For any graph $G$, if the policeman can always end in the same vertex as the robber, we call $G$ \emph{cop-win}. However, $G$ is \emph{$CECC$} if it is not cop-win, but any $G$-contraction is cop-win. The characteristics of a $CECC$ graph are studied in \cite{cameron2015edge}.

A further example is the investigation of the so-called \emph{contraction critical}, with respect to independence number, that is, an edge $e$ in a graph $G$ where $\alpha(G/e) \leq \alpha(G)$, studied in \cite{plummer2014note}. Furthermore, the case where $c$ is the chromatic and clique number, respectively, has been investigated in \cite{diner2018contraction, paulusma2016reducing, paulusma2019critical}.

In this article, we investigate the graph invariant $H$-free for a given set of graphs $\mathcal{H}$. In particular, we present sufficient and necessary conditions for a graph $G$ such that any $G$-contraction is $\mathcal{H}$-free.

Let $\mathcal{H}$ be a set of graphs. The set of \emph{elementary (minimal)} graphs in $\mathcal{H}$, denoted by \emph{$\elm$($\mathcal{H}$)}, is defined as $\{H \in \mathcal{H}: $ if $ G \in \mathcal{H} $ and $H$ is $G$-exist, then $G$ is isomorphic to $H\}$.
From the previous definition, we can directly obtain the following.
\begin{proposition}{\label{prop: elem}}
	Let $\mathcal{H}$ be a set of graphs. Graph $G$ is $\mathcal{H}$-free if and only if G is $\elm(\mathcal{H})$-free.
\end{proposition}

We call an $\mathcal{H}$-free graph $G$, \emph{strongly $\mathcal{H}$-free} if any $G$-contraction is $\mathcal{H}$-free. Furthermore, an $\mathcal{H}$-exist graph $G$ is a \emph{critically $\mathcal{H}$-exist} if any $G$-contraction is $\mathcal{H}$-free.
If we add any number of isolated vertices to a strongly $H$-free or critically $H$-exist graph, we obtain a graph with same property. Thus, from this section and forward, we exclude graphs having isolated vertices unless otherwise stated.

We conclude directly the following.
\begin{proposition}{\label{Theore: forbidden}}
	Let $\mathcal{H}$ be a set of graphs and $G$ be a graph where $G$ is neither critically $\mathcal{H}$-exist nor $\mathcal{H}$-free but not strongly $\mathcal{H}$-free. The graph $G$ is $\mathcal{H}$-free if and only if any $G$-contraction is $\mathcal{H}$-free.
\end{proposition}

Given a graph $G$ and a set of graphs $\mathcal{H}$, we call $G$ \emph{$\mathcal{H}$-split} if there is a $G$-contraction isomorphic to a graph in $\mathcal{H}$. Furthermore, $G$ is \emph{$\mathcal{H}$-free-split} if $G$ is $\mathcal{H}$-split and $\mathcal{H}$-free. Moreover, the set of all $\mathcal{H}$-free-split graphs, for a given $\mathcal{H}$, is denoted by \emph{$\fs(\mathcal{H})$}.

\begin{proposition}{\label{Theorem: key}}
	Let $\mathcal{H}$ be a set of graphs and $G$ be a $\mathcal{H}$-free graph. Then $G$ is strongly $\mathcal{H}$-free if and only if $G$ is $\fs(\mathcal{H})$-free.
\end{proposition}
\begin{proof}
	Assume for the sake of contradiction that there exists a strongly $\mathcal{H}$-free graph $G$ with an induced $\mathcal{H}$-free-split subgraph $J$. Consequently, there is an edge $e$ in $J$ such that $J/e$ induces a graph in $\mathcal{H}$. As a result, $G/e$ is $\mathcal{H}$-exist, which contradicts the fact that $G$ is strongly $\mathcal{H}$-free. 
	
	In contrast, if $G$ is an $\mathcal{H}$-free but not a strongly $\mathcal{H}$-free, then there is a set $U \subseteq V(G)$ such that there is an edge $e \in E(G[U])$ where $G/e$ is $\mathcal{H}$-exist. Let $U$ be a minimum set with such a property. Thus $G[U]$ is $\mathcal{H}$-free-split.
\end{proof}

From Propositions \ref{Theore: forbidden} and \ref{Theorem: key}, we deduce the following.
\begin{theorem}{\label{Theorem: characeterization}}
	Let $\mathcal{H}$ be a set of graphs and $G$ be a $\fs(\mathcal{H})$-free graph where $G$ is not critically $\mathcal{H}$-exist. The graph $G$ is $\mathcal{H}$-free if and only if any $G$-contraction is $\mathcal{H}$-free.
\end{theorem}

Theorem \ref{Theorem: characeterization} provides a sufficient and necessary condition that answers the question we investigate in this thesis, however, it translates the problem to determining characterizations for critically $\mathcal{H}$-exist and $\mathcal{H}$-free-split graphs for a set of graphs $\mathcal{H}$. In Sections \ref{Section: CC(H)} and \ref{Section: EC(H)}, we present some properties for these families of graphs.

\subsection{The \texorpdfstring{$\mathcal{H}$}{H}-Split Graphs}{\label{Section: CC(H)}}
Let $H$ be a graph with $v \in V(H)$ and $N_{H}(v) = U \cup W$. The \emph{$splitting(H,v,U,W)$} is the graph obtained from $H$ by removing $v$ and adding two vertices $u$ and $w$ where $ N_{H}(u)= U \cup \{w\}$ and $ N_{H}(w)= W \cup \{u\}$. Furthermore, \emph{$splitting(H,v)$} is the set of all graphs for any possible $U$ and $W$. Moreover, \emph{$splitting(H)$} is the union of the $splitting(H,v)$ for any vertex $v \in V(H)$. Given a set of graphs $\mathcal{H}$, \emph{$splitting(\mathcal{H})$} is the union of the splittings of every graph in $\mathcal{H}$.

\begin{theorem}{\label{Theorem: CC=H[]}}
	For a graph $G$ and a set of graphs $\mathcal{H}$, $G$ is an $\mathcal{H}$-split if and only if $G \in splitting(\mathcal{H})$.
\end{theorem}
\begin{proof}
	Let $G$ be an $\mathcal{H}$-split. Hence there is a graph $H \in \mathcal{H}$ such that $G$ is $H$-split.
	Thus, there are two vertices $u,w \in V(G)$ such that $G/uw$ is isomorphic to $H$. Let $x := V(G/uw)-V(G)$, then $N_{G/uw}(x) = (N_{G}(u) \cup N_{G}(w)) - \{u,w\}$. As a result, $G \in splitting(H,x)$. Consequently, $G \in splitting(\mathcal{H})$.

	Conversely, let $G \in splitting(\mathcal{H})$. Hence there is a graph $H \in \mathcal{H}$ such that $G \in splitting(H)$.
	Thus, there are two adjacent vertices $u,w \in V(G)$ such that $G/uw \cong H$. Thus, $G$ is $\mathcal{H}$-split.
\end{proof}

For a set of graphs $\mathcal{H}$ and using Theorem \ref{Theorem: CC=H[]}, we can use $splitting(\mathcal{H})$ to construct all $\mathcal{H}$-split graphs, consequently $\mathcal{H}$-free-split graphs.

\begin{proposition}
	In a graph $G$, let $u,v \in V(G)$. If $u$ is similar to $v$, then $splitting(G,u) = splitting(G,v)$.
\end{proposition}
By the previous proposition, for a graph $H$, the steps to construct the $H$-free-split graphs are:
\begin{itemize}
	\item Let $\pi$ be the partition of $V(H)$ induced by the orbits generated from $Aut(H)$;
	\item for every orbit $o \in \pi$, we choose a vertex $v \in o$; and
	\item construct $splitting(H,v)$.
\end{itemize}

\begin{proposition}\label{pro: split graphs that are not free}
	Let $G$ be a graph, $v$ a vertex in $V(G)$ where $N_{G}(v) = U \cup W$. If $U=N_{G}(v)$ or $W=N_{G}(v)$, then $splitting(G,v,U,W)$ is not $G$-free-split.
\end{proposition}
\begin{proposition}\label{pro: split graphs that are not free with degree one}
	Let $G$ be a graph and $v$ a vertex in $V(G)$. If $deg(v)=1$, then $splitting(G,v)$ contains no $G$-free-split graph.
\end{proposition}

\begin{proposition}\label{pro: no free-split for paths}
	If $G$ is a path, then $splitting(G)$ contains no $G$-free-split graph.
\end{proposition}

\begin{proposition}\label{pro: only one free-split for cycles}
	If $G$ is a $C_{n}$ for an integer $n \geq 3$, then the $G$-free-split is $C_{n+1}$.
\end{proposition}

\subsection{Critically \texorpdfstring{$\mathcal{H}$}{H}-Exist Graphs}{\label{Section: EC(H)}}

\begin{theorem}{\label{Theorem: critical H-exist: not S is independent}}
	Let $G$ be a graph and $\mathcal{H}$ be a set of graphs.
	If $G$ is a critically $\mathcal{H}$-exist, then for any $S \subseteq V(G)$ such that $G[S]$ is isomorphic to a graph in $\mathcal{H}$, the followings properties hold:
	\begin{enumerate}
		\item $V(G) - S$ is independent and
		\item there is no corner in $V(G) - S$ that is dominated by a vertex in $S$.
	\end{enumerate}
\end{theorem}
\begin{proof}
	\begin{enumerate}
		\item For the sake of contradiction, assume there is a $S \subseteq V(G)$ such that $G[S]$ is isomorphic to a graph $H \in \mathcal{H}$  but $V(G) - S$ is not independent. Hence, there are two vertices $u,v \in V(G) - S$ where $u$ and $v$ are adjacent. Consequently, $G/uv[S]$ is isomorphic to $H$, which contradicts the fact that $G$ is a critically $\mathcal{H}$-exist.

		\item Since $V(G) - S$ is independent, the neighborhood of any vertex in $V(G) - S$ is a subset of $S$. For the sake of contradiction, assume that there is a corner $u \in V(G) - S$ that is dominated by $v \in S$. However, $G/uv[S]$ is isomorphic to a graph $H \in \mathcal{H}$, which contradicts the fact that $G$ is a critically $\mathcal{H}$-exist.
	\end{enumerate}
\end{proof}

\begin{corollary}\label{coro: no vertex adajcet to 1 or 2 or 3 or max degree in critical graph}
	Let $G$ be a critically $\mathcal{H}$-exist graph for a set of graphs $\mathcal{H}$. If $S$ is a vertex set that induces a graph in $\mathcal{H}$, then no vertex in $V(G)-S$ is adjacent to exactly one vertex, two adjacent vertices, three vertices that induce either $P_{3}$ or $C_{3}$, or a vertex with degree $|V(G)|-1$.
\end{corollary}

Let $G$ be a graph with adjacent vertices $u,v$, and $\{w\} := V(G/uv) - V(G)$.
We define the mapping $f:2^{V(G)} \to 2^{V(G/uv)}$ as follows:
\[
f(S) = 
\begin{cases}
	S & \text{if }u,v \notin S,\\
	(S \cup \{w\}) - \{u,v\} & \text{otherwise.} \\ 
\end{cases}
\]
Let $S$ be a vertex set such that $G[S]$ is isomorphic to a given graph $H$. We call an edge $uv$, $H$-critical for $S$ if $G/uv[f(S)]$ is non-isomorphic to $H$. Furthermore, we call the edge $uv$ $H$-critical in $G$ if for any vertex subset $S$ that induces $H$, $uv$ is $H$-critical for $S$.

\begin{theorem}{\label{Theorem: constructive edge characterization}}
	Let $G$ be a graph and $S \subseteq V(G)$ where $H$ is the graph induced by $S$ in $G$. For any edge $uv \in E(G)$, $uv$ is $H$-critical for $S$ if and only if 
	\begin{enumerate}
		\item $	u,v \in S$ or
		\item $u \in V(G) - S$, $v \in S$, and $u$ is not a corner dominated by $v$ in the subgraph $G[S\cup \{u\}]$.%$N_{S}(v)- \{u\} \not\subseteq N_{S}(u)$.
	\end{enumerate}
\end{theorem}
\begin{proof}
	\begin{enumerate}
		\item If $u,v \in S$, then $|f(S)| < |S|$. Thus, $G/uv[f(S)]$ is non-isomorphic to $H$.
		\item Let $u \in V(G) - S$, $v \in S$, and $u$ is not a corner dominated by $v$ in the subgraph $G[S\cup \{u\}]$. Additionally, let $w \in N_{S}(u)$ but $w \notin N_{S}(v)$. In $G/uv$, let $x := V(G/uv) - V(G)$. Clearly, $x$ is adjacent to any vertex in $N_{S}(v) \cup \{w\}$. Hence, the size of $G/uv[f(S)]$ is larger than that of $G[S]$. Thus, $G/uv[f(S)]$ is non-isomorphic to $H$.
	\end{enumerate}
	
	Conversely, if none of the conditions in the theorem hold, then one of the following holds:
	\begin{enumerate}
		\item both $u$ and $v$ are not in $S$, or
		\item $u \in V(G) - S$, $v \in S$, and $u$ is a corner dominated by $v$ in the subgraph $G[S\cup \{u\}]$.
	\end{enumerate}
	In both cases, $G[S] \cong G/uv[f(S)] \cong H$. Consequently, $uv$ is not $H$-critical for $S$.
\end{proof}

\section{Special graphs}{\label{Section: application}}
\begin{proposition}{\label{Lemma: Cycle contraction}}
	If a graph $G$ is $C_{n}$-exist, where $n \geq 4$, then there is a $G$-contraction that is $C_{n-1}$-exist. 
\end{proposition}
\begin{proposition}{\label{Lemma: EC(C3)}}
	The only critical $C_{3}$-exist graph is $C_{3}$.
\end{proposition}

\section{Claw-Free Graphs}
There are several graph families that are subfamilies of claw-free graphs, for instance, line graphs and complements of triangle-free graphs. For more graph families and results about claw-free graphs, please consult \cite{faudree1997claw}. Additionally, for more structural results about claw-free graphs, please consult \cite{chudnovsky2008claw,chudnovsky2005structure}. In the following, we call the graph $H_{5}$ in Figure \ref{Figure: the graphs in CC(claw)} bull.

\begin{proposition}{\label{Proposition: CC(claw)}}
	The graphs in Figure \ref{Figure: the graphs in CC(claw)} are the only claw-split graphs.
\end{proposition}
\begin{corollary}{\label{Corollary: FCC(claw)}}
	Bull is the only claw-free-split graph.
\end{corollary}
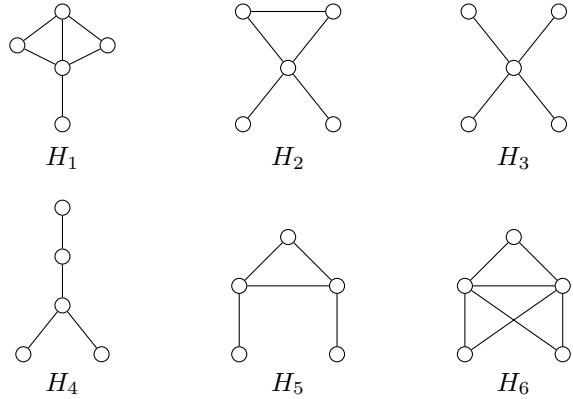
\begin{figure}[ht!]
	\centering
	\begin{tikzpicture}[hhh/.style={draw=black,circle,inner sep=2pt,minimum size=0.2cm}]
		\begin{scope}[shift={(-3,0)},scale=1.5]
			\node 	   (h) at (0,0) 	{$H_{1}$};
			\node[hhh] (e) at (0,0.3) 	{};
			\node[hhh] (a) at (0,0.8) 	{};
			\node[hhh] (b) at (0.4,1) 	{};
			\node[hhh] (c) at (0,1.3) 	{};
			\node[hhh] (d) at (-0.4,1) 	{};
			
			\draw (c) -- (a) -- (b)-- (c) -- (d) -- (a) --(e) ;
		\end{scope}
		
		\begin{scope}[shift={(0,0)},scale=1.5]
			\node 	(h) at (0,0) 	{$H_{2}$};
			\node[hhh] 	(e) at (-0.4,0.3) 	{};
			\node[hhh] 	(d) at (0.4,0.3) 	{};
			\node[hhh] 	(a) at (0,0.8) 	{};
			\node[hhh]  (b) at (-0.4,1.3) 	{};
			\node[hhh] 	(c) at (0.4,1.3) 	{};
			
			\draw (a) -- (c) -- (b)-- (a) -- (d)  (e) --(a);
		\end{scope}
		
		\begin{scope}[shift={(3,0)},scale=1.5]
			\node 	(h) at (0,0) 	{$H_{3}$};
			\node[hhh] (d) at (-0.4,0.3) 	{};
			\node[hhh] (e) at (0.4,0.3) 	{};
			\node[hhh] 	(a) at (0,0.8) 	{};
			\node[hhh] (b) at (-0.4,1.3) 	{};
			\node[hhh] (c) at (0.4,1.3) 	{};
			
			\draw (c) -- (a) -- (b)  (d) -- (a) --(e);
		\end{scope}
		
		\begin{scope}[shift={(-3,-3)},scale=1.3]
			\node 	(h) at (0,0) 	{$H_{4}$};
			\node[hhh] 	(a) at (0,0.8) 	{};
			\node[hhh]  (b) at (0,1.3) 	{};
			\node[hhh] 	(c) at (0,1.8) 	{};
			\node[hhh] 	(d) at (-0.4,0.3) 	{};
			\node[hhh] 	(e) at (0.4,0.3) 	{};
			
			\draw (c) -- (b)-- (a) -- (d)  (e) --(a);
		\end{scope}
		
		\begin{scope}[shift={(0,-3)},scale=1.3]
			\node 	(h) at (0,0) 	 {$H_{5}$};
			\node[hhh] 	(a) at (0,1.5) 		{};
			\node[hhh]  (b) at (-0.5,1) 	{};
			\node[hhh] 	(c) at (-0.5,0.3) 	{};
			\node[hhh] 	(d) at (0.5,1) 		{};
			\node[hhh] 	(e) at (0.5,0.3) 	{};
			
			\draw (a) -- (b) --(d) --(a)   (c) -- (b) (d)--(e);
		\end{scope}
		
		\begin{scope}[shift={(3,-3)},scale=1.3]
			\node 	(h) at (0,0) 	{$H_{6}$};
			\node[hhh] 	(a) at (0,1.5) 	{};
			\node[hhh]  (b) at (-0.5,1) 	{};
			\node[hhh] 	(c) at (-0.5,0.3) 	{};
			\node[hhh] 	(d) at (0.5,1) 	{};
			\node[hhh] 	(e) at (0.5,0.3) 	{};
			
			\draw (a) -- (b) --(d) --(a)  (d) -- (c) -- (b) (d) -- (e) -- (b);
		\end{scope}
	\end{tikzpicture}	
	\caption{Claw-split graphs}
	\label{Figure: the graphs in CC(claw)}
\end{figure}

\begin{proposition}{\label{Proposition: EC(claw)}}
	The graphs in Figure \ref{Figure: the graphs in EC(claw)} are the only critically claw-exist graphs.
\end{proposition}
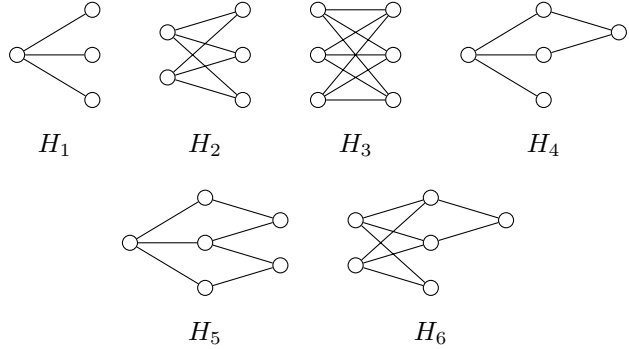
\begin{figure}[ht!]
	\centering
	\begin{tikzpicture}[hhh/.style={draw=black,circle,inner sep=2pt,minimum size=0.2cm}]
		\begin{scope}[shift={(-3.5,0)},scale=2]
			\node 	   (a) at (-0.25,0) 		{$H_{1}$};
			\node[hhh] (b) at (-0.5,0.6)	{};
			\node[hhh] (c1) at (0,0.3) 		{};
			\node[hhh] (c2) at (0,0.6) 		{};
			\node[hhh] (c3) at (0,0.9) 		{};
			
			\draw (c1) -- (b) (c2) -- (b)  (c3) -- (b);
		\end{scope}
		
		\begin{scope}[shift={(-1.5,0)},scale=2]
			\node 	   (a) at (-0.25,0) 		{$H_{2}$};
			\node[hhh] (b1) at (-0.5,0.75)	{};
			\node[hhh] (b2) at (-0.5,0.45)	{};
			\node[hhh] (c1) at (0,0.3) 		{};
			\node[hhh] (c2) at (0,0.6) 		{};
			\node[hhh] (c3) at (0,0.9) 		{};
			
			\draw (c1) -- (b1) (c2) -- (b1)  (c3) -- (b1)
			(c1) -- (b2) (c2) -- (b2)  (c3) -- (b2);
		\end{scope}
		
		\begin{scope}[shift={(0.5,0)},scale=2]
			\node 	   (a) at (-0.25,0) 		{$H_{3}$};
			\node[hhh] (b1) at (-0.5,0.3)	{};
			\node[hhh] (b2) at (-0.5,0.6)	{};
			\node[hhh] (b3) at (-0.5,0.9)	{};
			\node[hhh] (c1) at (0,0.3) 		{};
			\node[hhh] (c2) at (0,0.6) 		{};
			\node[hhh] (c3) at (0,0.9) 		{};
			
			\draw (c1) -- (b1) (c2) -- (b1)  (c3) -- (b1)
			(c1) -- (b2) (c2) -- (b2)  (c3) -- (b2)
			(c1) -- (b3) (c2) -- (b3)  (c3) -- (b3);
		\end{scope}
		
		\begin{scope}[shift={(2.5,0)},scale=2]
			\node 	   (a) at (0,0) 		{$H_{4}$};
			\node[hhh] (b) at (-0.5,0.6)	{};
			\node[hhh] (c1) at (0,0.3) 		{};
			\node[hhh] (c2) at (0,0.6) 		{};
			\node[hhh] (c3) at (0,0.9) 		{};
			\node[hhh] (d) at (0.5,0.75) 	{};
			
			\draw (c1) -- (b) (c2) -- (b)  (c3) -- (b) (c2)--(d)--(c3);
		\end{scope}
		
		\begin{scope}[shift={(-2,-2.5)},scale=2]
			\node 	   (a) at (0,0) 		{$H_{5}$};
			\node[hhh] (b) at (-0.5,0.6)	{};
			\node[hhh] (c1) at (0,0.3) 		{};
			\node[hhh] (c2) at (0,0.6) 		{};
			\node[hhh] (c3) at (0,0.9) 		{};
			\node[hhh] (d1) at (0.5,0.75) 	{};
			\node[hhh] (d2) at (0.5,0.45) 	{};
			
			\draw (c1) -- (b) (c2) -- (b)  (c3) -- (b) 
			(c2)--(d1)--(c3) (c1)--(d2)--(c2);
		\end{scope}
		
		\begin{scope}[shift={(1,-2.5)},scale=2]
			\node 	   (a) at (0,0) 		{$H_{6}$};
			\node[hhh] (b1) at (-0.5,0.75)	{};
			\node[hhh] (b2) at (-0.5,0.45) 	{};
			\node[hhh] (c1) at (0,0.3) 		{};
			\node[hhh] (c2) at (0,0.6) 		{};
			\node[hhh] (c3) at (0,0.9) 		{};
			\node[hhh] (d) at (0.5,0.75) 	{};

			\draw (c1) -- (b1) (c2) -- (b1)  (c3) -- (b1) 
			(c1) -- (b2) (c2) -- (b2)  (c3) -- (b2) 
			(c2)--(d)--(c3);
		\end{scope}
	\end{tikzpicture}	
	\caption{Critically claw-exist graphs}
	\label{Figure: the graphs in EC(claw)}
\end{figure}

\begin{proof}
	\renewcommand{\qedsymbol}{\claimqed}
	Through this proof, we assume that $G$ is a critically claw-exist graph with $S := \{r,s,t,u\}$, where $G[S]$ is isomorphic to a claw and $u$ is its center. By Theorem \ref{Theorem: critical H-exist: not S is independent}, $V(G) - S$ is independent. Thus, any vertex in $V(G) - S$ is adjacent to vertices only in $S$. By Corollary \ref{coro: no vertex adajcet to 1 or 2 or 3 or max degree in critical graph}, if $v \in V(G) - S$, then neither $|N(v)| = 1$ nor $v$ is adjacent to $u$.
	
	Let $v,w \in V(G) - S$ such that $N(v) = N(w)$ where $|N(v)|=2$.
	W.l.o.g., assume that $N(v) =\{r,s\}$, however, in $G/tu$, $f(\{r,u,v,w\})$ induces a claw, which contradicts the fact that $G$ is a critically claw-exist. Thus, if $v,w \in V(G) - S$ such that $|N(v)| = |N(w)|=2$, then $N(v) \not= N(w)$.
	
	Let $v,w,x \in V(G) - S$ such that $|N(v)| = |N(w)| = 2$.
	No two vertices of $v,w$, and $x$ (if $|N(x)| =2$) are adjacent to the same vertices in $S$. W.l.o.g., assume that $N(v) =\{r,s\}$, $N(w) =\{r,t\}$, and $ \{s,t\} \subseteq N(x)$. In $G/tx$, $f(\{r,u,v,w\})$ induces a claw, which contradicts the fact that $G$ is a critically claw-exist. Thus, if $v,w \in V(G) - S$ such that $|N(v)| = |N(w)|=2$, then $G$ is isomorphic to $H_{5}$.

	Let $v,w,x \in V(G) - S$ such that $|N(v)| = |N(w)| = 3$. 
	W.l.o.g., assume that $s$ is adjacent to $x$. In $G/sx$, $f(\{r,u,v,w\})$ induces a claw, which contradicts the fact that $G$ is a critically claw-exist. Thus, if $v,w \in V(G) - S$ such that $N(v) = N(w)$ and $|N(v)|=3$, then $G$ is isomorphic to $H_{3}$.

	Consequently, the possible critically claw-exist graphs are those presented in Figure \ref{Figure: the graphs in EC(claw)}. To complete the proof, we have to show that all these graphs are critically claw-exist, which is straightforward in each case.
	\renewcommand{\qedsymbol}{$\square$}
\end{proof}

By Theorem \ref{Theorem: characeterization}, Corollary \ref{Corollary: FCC(claw)}, and Proposition \ref{Proposition: EC(claw)}, we obtain the following result.
\begin{theorem}{\label{Theorem: claw-free characterization}}
	Let $G$ be a bull-free graph that is non-isomorphic to any graph in Figure \ref{Figure: the graphs in EC(claw)}. The graph $G$ is claw-free if and only if any $G$-contraction is claw-free.
\end{theorem}

\section{The \texorpdfstring{$2K_{2}$}{2K2}-Free Graphs}
Different graphs families are $2K_{2}$-free graphs; for instance split, pseudo-split, threshold, and co-chordal graphs. Various graph invariants were studied for $2K_{2}$-free graphs, please consult \cite{golan2016nonempty,2k2Free,el1985existence,brause2019chromatic,broersma2014toughness}. The class of $2K_{2}$-free graphs has been characterized in different ways, see \cite{meister2006two,on2k2graphs}. We call the graph $H_{5}$ in Figure \ref{Figure: the graphs in CC(claw)} Bull.

We call an edge $uv$ in a graph $G$ \emph{almost-dominating} if $V(G) - N[\{u,v\}]$ induces edgeless graph.
\begin{proposition}{\label{Proposition: almost-dominating and 2k2-free}}
	A graph $G$ is $2K_{2}$-free if and only if any edge in $E(G)$ is almost-dominating.
\end{proposition}
\begin{lemma}{\label{lemma: almost-dominating and 2k2-free}}
	Let $G$ be a graph with a unique subset $S \subseteq V(G)$ such that $G[S]$ induces $2K_{2}$. If every edge $e$ in $E(G)$ is $e$ is $2K_{2}$-critical for $S$, then $G$ is a critically $2K_{2}$-exist.  
\end{lemma}
\begin{proof}
	Let $H$ be a $G$-contraction. Every edge $e$ in $E(G)$ is $2K_{2}$-critical for $S$, then $V(G)-S$ is independent set. Furthermore, every vertex in $V(G)-S$ is adjacent to at least two nonadjacent vertices in $S$. In $H$, let $u \in V(G)-f(S)$ and $v \in f(S)$. If $u,v$ are adjacent, then $uv$ is almost-dominating. Let $u \in f(S)$, then $uv$ is almost-dominating. Hence, every edge in $H$ is almost-dominating. Thus, $G$ is a critically $2K_{2}$-exist. 
\end{proof}

\begin{proposition}{\label{Proposition: CC(2K2)}}
	The graphs $P_{2} \cup C_{3}$ and $P_{2} \cup P_{3}$ are the only $2K_{2}$-split graphs.
\end{proposition}
\begin{corollary}{\label{Corollary: FCC(2K2)}}
	There is no $2K_{2}$-free-split graph.
\end{corollary}

\begin{proposition}{\label{Proposition: EC(2k2)}}
	The graphs in Figure \ref{Figure: the graphs in EC(2K2)} are the only critically $2K_{2}$-exist graphs. 
\end{proposition}

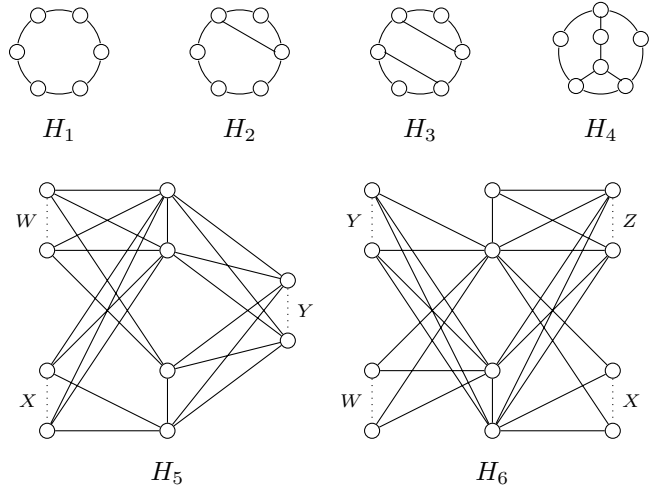
\begin{figure}
	\centering
	\begin{tikzpicture}[hhh/.style={draw=black,circle,inner sep=2pt,minimum size=0.2cm},scale=0.8]
		\begin{scope}[shift={(-4.5,0)}]
			\node 		(h) at (0,0)	 	{$H_{1}$};
			\begin{scope}[shift={(0,1.3)}]
				\def \n {6}
				\def \radius {0.7cm}
				\def \radiusCorrect {3}
				\def \margin {12} % margin in angles, depends on the radius
				\def \rotate {0}	% to rotate the cycle
				
				\foreach \s in {1,...,\n}
				{
					\node[hhh] at ({360/\n * (\s - 1)+\rotate}:\radius) {};
					\draw[ >=latex]  ({360/\n * (\s - 1)+\margin + \rotate}:\radius)
					arc ({360/\n * (\s - 1)+\margin+\rotate}:{360/\n * (\s)-\margin + \rotate}:\radius);
				}
			\end{scope}
		\end{scope}

		\begin{scope}[shift={(-1.5,0)}]
			\node 		(h) at (0,0)	 	{$H_{2}$};
			\begin{scope}[shift={(0,1.3)}]
				\def \n {6}
				\def \radius {0.7cm}
				\def \radiusCorrect {3}
				\def \margin {12} % margin in angles, depends on the radius
				\def \rotate {0}	% to rotate the cycle
				
				\foreach \s in {1,...,\n}
				{
					\node[hhh] at ({360/\n * (\s - 1)+\rotate}:\radius) {};
					\draw[ >=latex]  ({360/\n * (\s - 1)+\margin + \rotate}:\radius)
					arc ({360/\n * (\s - 1)+\margin+\rotate}:{360/\n * (\s)-\margin + \rotate}:\radius);
				}
				\draw ({360/\n * (1 - 1) + \rotate}:\radius-\radiusCorrect)  --  ({360/\n * (3 - 1) + \rotate}:\radius-\radiusCorrect);
			\end{scope}
		\end{scope}	
		
		\begin{scope}[shift={(1.5,0)}]
			\node 		(h) at (0,0)	 	{$H_{3}$};
			\begin{scope}[shift={(0,1.3)}]
				\def \n {6}
				\def \radius {0.7cm}
				\def \radiusCorrect {3}
				\def \margin {12} % margin in angles, depends on the radius
				\def \rotate {0}	% to rotate the cycle
				
				\foreach \s in {1,...,\n}
				{
					\node[hhh] at ({360/\n * (\s - 1)+\rotate}:\radius) {};
					\draw[ >=latex]  ({360/\n * (\s - 1)+\margin + \rotate}:\radius)
					arc ({360/\n * (\s - 1)+\margin+\rotate}:{360/\n * (\s)-\margin + \rotate}:\radius);
				}
				\draw ({360/\n * (1 - 1) + \rotate}:\radius-\radiusCorrect)  --  ({360/\n * (3 - 1) + \rotate}:\radius-\radiusCorrect)
				({360/\n * (6 - 1) + \rotate}:\radius-\radiusCorrect)  --  ({360/\n * (4 - 1) + \rotate}:\radius-\radiusCorrect);
			\end{scope}
		\end{scope}
		
		\begin{scope}[shift={(4.5,0)}]
			\node 		(h) at (0,0)	 	{$H_{4}$};
			\begin{scope}[shift={(0,1.3)}]
				\def \n {5}
				\def \radius {0.7cm}
				\def \radiusCorrect {3}
				\def \margin {12} % margin in angles, depends on the radius
				\def \rotate {90}	% to rotate the cycle
				
				\foreach \s in {1,...,\n}
				{
					\node[hhh] at ({360/\n * (\s - 1)+\rotate}:\radius) {};
					\draw[ >=latex]  ({360/\n * (\s - 1)+\margin + \rotate}:\radius)
					arc ({360/\n * (\s - 1)+\margin+\rotate}:{360/\n * (\s)-\margin + \rotate}:\radius);
				}
				\node[hhh] (x) at (0,0.25) 	{};
				\node[hhh] (t) at (0,-0.25) 	{};
				\draw ({360/\n * (1 - 1) + \rotate}:\radius-\radiusCorrect)  --  (x) -- (t)
				({360/\n * (4 - 1) + \rotate}:\radius-\radiusCorrect)  -- (t) --  ({360/\n * (3 - 1) + \rotate}:\radius-\radiusCorrect);
			\end{scope}
		\end{scope}

		\begin{scope}[shift={(-2.7,-6)}]
			\node 	   (label) at (0,0.3) 		{$H_{5}$};
			\node[hhh] (r) at (0,1) 	{};
			\node[hhh] (s) at (0,2) 	{};
			\node[hhh] (t) at (0,4) 	{};
			\node[hhh] (u) at (0,5) 	{};
			
			\node[hhh] (v41) at (2,3.5) 	{};
			\node[hhh] (v42) at (2,2.5) 	{};
			%\node (v4dots) at (2,3.15) 	{$\vdots Y$};
			\draw[dotted] (v41) to node[right]{\scriptsize $Y$} (v42);
			
			\node[hhh] (v311) at (-2,5) 	{};
			\node[hhh] (v312) at (-2,4) 	{};
			%	\node (v4dots) at (-2,4.6) 	{$W \vdots$};
			\draw[dotted] (v311) to node[left]{\scriptsize $W$} (v312);	
			
			\node[hhh] (v321) at (-2,2) 	{};
			\node[hhh] (v322) at (-2,1) 	{};
			%\node (v4dots) at (-2,1.6) 	{$X \vdots$};		
			\draw[dotted] (v321) to node[left]{\scriptsize $X$} (v322); 
			
			\draw (r) -- (s) (u) -- (t)
			(r) --(v41) -- (s) (t) --(v41) -- (u)
			(r) --(v42) -- (s) (t) --(v42) -- (u)
			
			(u) --(v311) -- (t) (s) --(v311)
			(u) --(v312) -- (t) (s) --(v312)
			
			(u) --(v321) -- (t) (r) --(v321)
			(u) --(v322) -- (t) (r) --(v322);
		\end{scope}

		\begin{scope}[shift={(2.7,-6)}]
			\node 	   (label) at (0,0.3) 		{$H_{6}$};
			\node[hhh] (r) at (0,1) 	{};
			\node[hhh] (s) at (0,2) 	{};
			\node[hhh] (t) at (0,4) 	{};
			\node[hhh] (u) at (0,5) 	{};
			
			\node[hhh] (v41) at (2,5) 	{};
			\node[hhh] (v42) at (2,4) 	{};
			%	\node (v4dots) at (2,4.6) 	{$\vdots$};
			\draw[dotted] (v41) to node[right]{\scriptsize $Z$} (v42);
			
			\node[hhh] (v31) at (-2,5) 	{};
			\node[hhh] (v32) at (-2,4) 	{};
			%\node (v3dots) at (-2,4.6) 	{$\vdots$};	
			\draw[dotted] (v31) to node[left]{\scriptsize $Y$} (v32);
			
			\node[hhh] (v211) at (-2,2) 	{};
			\node[hhh] (v212) at (-2,1) 	{};
			%	\node (v21dots) at (-2,1.6) 	{$\vdots$};	
			\draw[dotted] (v211) to node[left]{\scriptsize $W$} (v212);
			
			\node[hhh] (v221) at (2,2) 	{};
			\node[hhh] (v222) at (2,1) 	{};
			%	\node (v22dots) at (2,1.6) 	{$\vdots$};	
			\draw[dotted] (v221) to node[right]{\scriptsize $X$} (v222);	
			
			\draw (r) -- (s) (u) -- (t)
			(r) --(v41) -- (s) (t) --(v41) -- (u)
			(r) --(v42) -- (s) (t) --(v42) -- (u)
			
			(r) --(v31) -- (s) (t) --(v31)
			(r) --(v32) -- (s) (t) --(v32)
			
			(s) --(v211) -- (t)	(s) --(v212) -- (t)
			(r) --(v221) -- (t)	(r) --(v222) -- (t);
		\end{scope}
		
	\end{tikzpicture}
	\caption{Critically $2K_{2}$-exist graphs}
	\label{Figure: the graphs in EC(2K2)}
\end{figure}

\begin{proof}
	\renewcommand{\qedsymbol}{\claimqed}
	Through this proof, we assume that $G$ is a critically $2K_{2}$-exist graph with $S=\{r,s,t,u\}$ such that $G[S]$ is isomorphic to $2K_{2}$, where $rs$ and $tu$ are edges in $G$. By Theorem \ref{Theorem: critical H-exist: not S is independent}, we note that $V(G) - S$ is independent. Thus, any vertex in $V(G) - S$ is adjacent to vertices only in $S$. By Corollary \ref{coro: no vertex adajcet to 1 or 2 or 3 or max degree in critical graph}, if $v \in V(G) - S$, then neither $|N(v)| = 1$ nor $v$ is adjacent to exactly two adjacent vertices.
	
	\begin{claim}{\label{Claim: EC(2K2) different non intersecting v2}}
		If $v,w \in V(G) - S$ such that $|N(v)| = |N(w)|=2$ while $N(v) \cap N(w) = \phi$, then $G$ is isomorphic to $H_{1}$.
	\end{claim}
	\begin{proof}
		W.l.o.g., let $N(v) = \{r,u\}$ and $N(w) = \{s,t\}$. We will show that $V(G)=S\cup \{v,w\}$. For the sake of contradiction, assume that there is a vertex $x \in V(G) - S$. Thus, $x$ is adjacent to at least one vertex in $S$. W.l.o.g., let $x$ be adjacent to $r$. In $G/rx$, $f(\{s,u,v,w\})$ induces $2K_{2}$, which contradicts the fact that $G$ is a critically $2K_{2}$-exist. 
	\end{proof}
	
	\begin{claim}{\label{Claim: EC(2K2) different non intersecting v2 v3}}
		If $v,w,x \in V(G) - S$ such that $|N(v)| = |N(w)|=2$, $|N(x)| =3$ while $N(v) = N(w)$, then $N(v) \subset N(x)$.
	\end{claim}
	\begin{proof}
		W.l.o.g., let $N(v) =N(w)= \{r,u\}$. For the sake of contradiction and w.l.o.g., assume $N(x) = \{r,s,t\}$. In $G/rw$, $f(\{s,u,v,x\})$ induces $2K_{2}$, which contradicts the fact that $G$ is a critically $2K_{2}$-exist. 
	\end{proof}

	\begin{claim}{\label{Claim: EC(2K2) v2,v2,v3 maximal}}
		If $v,w,x \in V(G) - S$ such that $|N(v)| = |N(w)| = 2$ while $|N(v) \cap N(w)| = 1$ and $|N(x)| = 3$ where $N(v) \cap N(w) \cap N(x) = \phi$, then $G$ is isomorphic to $H_{4}$.
	\end{claim}
	\begin{proof}
		W.l.o.g., let $N(v) = \{r,u\}$, $N(w) = \{r,t\}$, and $N(x) = \{s,t,u\}$. For the sake of contradiction, assume that there is a vertex $y \in V(G) - S$. Hence, $y$ is adjacent to at least one vertex in $S$. If $y$ is adjacent to $s$ (or $u$), then $f(\{r,t,v,x\})$ induces $2K_{2}$ in $G/sy$ (or $G/uy$), which contradicts the fact that $G$ is a critically $2K_{2}$-exist. Moreover, if $y$ is adjacent to $t$, then $f(\{r,u,w,x\})$ induces $2K_{2}$ in $G/ty$, which contradicts the fact that $G$ is a critically $2K_{2}$-exist.
	\end{proof}
	
	\begin{claim}{\label{Claim: EC(2K2) v2,v2,v3 maximal 2}}
		If $v,w,x \in V(G) - S$ such that $|N(v)|=|N(w)|=2$ while $|N(v) \cap N(w)| = 1$ and $|N(x)| = 3$, then either $N(v) \cup N(w) =  N(x)$ or $N(v) \cap N(w) \cap N(x) = \phi$ and $G$ is isomorphic to $H_{4}$.
	\end{claim}
	\begin{proof}
		By Claim \ref{Claim: EC(2K2) v2,v2,v3 maximal}, if $N(v) \cap N(w)\cap N(x)=\phi$, then $G$ is isomorphic to $H_{4}$. If $N(v)\cup N(w)= N(x)$, then we are done. As a result, and w.l.o.g, let $N(v) = \{r,u\}$ and $N(w)=\{r,t\}$. Assume for the sake of contradiction that $N(x)=\{r,s,t\}$. However, $f(\{s,u,v,x\})$ induces $2K_{2}$ in $G/rw$, which contradicts the fact that $G$ is a critically $2K_{2}$-exist.
	\end{proof}

	\begin{claim}{\label{Claim: EC(2K2) v3,v3 maximal}}
		If $v,w \in V(G) - S$ such that $|N(v)|=|N(w)|=3$, while $N(v) \cap N(w)$ consists of two nonadjacent vertices in $S$, then $G$ is isomorphic to $H_{3}$.
	\end{claim}
	\begin{proof}
		W.l.o.g., let $N(v)=\{r,t,u\}$ and $N(w)=\{r,s,t\}$. For the sake of contradiction, assume that there is a vertex $x \in V(G) - S$. If $x$ is adjacent to $r$ (or $t$), then $f(\{s,u,v,w\})$ induces $2K_{2}$ in $G/rx$ (or $G/tx$), which contradicts the fact that $G$ is a critically $2K_{2}$-exist graph. Thus, $N(x)=\{s,u\}$, however, $f(\{s,t,v,x\})$ induces $2K_{2}$ in $G/rw$, which contradicts the fact that $G$ is a critically $2K_{2}$-exist.
	\end{proof}

	\begin{claim}{\label{Claim: EC(2K2) no v3,v3,nv2}}
		If $v,w,x \in V(G) - S$ such that $|N(v)|=|N(w)|=3$, while $N(v) \cap N(w)$ is two adjacent vertices in $S$, then $|N(x)|\not=2$.
	\end{claim}
	\begin{proof}
		W.l.o.g., Let $N(v)=\{r,t,u\}$ and $N(w)=\{s,t,u\}$. W.l.o.g and for the sake of contradiction, assume that $N(x)=\{r,u\}$. However, $f(\{r,t,w,x\})$ induces $2K_{2}$ in $G/uv$, which contradicts the fact that $G$ is a critically $2K_{2}$-exist.
	\end{proof}	
	
	By Claims \ref{Claim: EC(2K2) different non intersecting v2} to \ref{Claim: EC(2K2) no v3,v3,nv2}, the possible critically $2K_{2}$-exist graphs are those presented in Figure \ref{Figure: the graphs in EC(2K2)} whose proofs of being critically $2K_{2}$-exist for $H_{1}$, $H_{2}$, $H_{3}$, and $H_{4}$ are straightforward.
	
	\begin{claim}{\label{Claim: EC(2K2) H5 is a critical 2k2-exist}}
		The graph $H_{5}$ in Figure \ref{Figure: the graphs in EC(2K2)} is a critically $2K_{2}$-exist.
	\end{claim}
	\begin{proof}
		Graph $H_{5}$ in Figure \ref{Figure: the graphs in EC(2K2)} is isomorphic to a graph $G$ that contains a vertex subset $S=\{r,s,t,u\}$, where $G[S]$ is isomorphic to a $2K_{2}$ and $rs,tu \in E(G)$. Moreover, $V(G)= S \cup W \cup X \cup Y$, such that $N(w \in W)=\{r,s,t\}$, $N(x \in X)=\{r,s,u\}$, $N(y \in Y)=\{r,s,t,u\}$, and $|W|,|X|,|Y| \geq 0$.
		
		We note that $S$ is the only vertex set inducing $2K_{2}$ in $G$. Moreover, every edge in $E( G )$ is $G[S]$-critical for $S$. Thus, and by Lemma \ref{lemma: almost-dominating and 2k2-free}, $H_{5}$ in Figure \ref{Figure: the graphs in EC(2K2)} is a critically $2K_{2}$-exist.
	\end{proof}
	
	\begin{claim}{\label{Claim: EC(2K2) H6 is a critical 2k2-exist}}
		Graph $H_{6}$ in Figure \ref{Figure: the graphs in EC(2K2)} is a critically $2K_{2}$-exist.
	\end{claim}
	\begin{proof}
		Graph $H_{6}$ in Figure \ref{Figure: the graphs in EC(2K2)} is isomorphic to a graph $G$ that contains a vertex subset $S=\{r,s,t,u\}$ where $G[S]$ is isomorphic to a $2K_{2}$ and $rs,tu \in E(G)$. Moreover, $V(G)= S \cup W \cup X \cup Y \cup Z$, such that $N(w \in W)=\{s,t\}$, $N(x \in X)=\{s,u\}$, $N(y)=\{s,t,u\}$, $N(z)=\{r,s,t,u\}$, and $|W|,|X|,|Y|,|Z| \geq 0$.
		
		We note that $S$ is the only vertex set inducing $2K_{2}$ in $G$. Moreover, every edge in $E( G )$ is $G[S]$-critical for $S$. Thus, and by Lemma \ref{lemma: almost-dominating and 2k2-free}, $H_{6}$ in Figure \ref{Figure: the graphs in EC(2K2)} is a critically $2K_{2}$-exist.
	\end{proof}
	By Claims \ref{Claim: EC(2K2) H5 is a critical 2k2-exist} and \ref{Claim: EC(2K2) H6 is a critical 2k2-exist}, the proof is complete.
	\renewcommand{\qedsymbol}{$\square$}
\end{proof}

By Theorem \ref{Theorem: characeterization}, Corollary \ref{Corollary: FCC(2K2)}, and Proposition \ref{Proposition: EC(2k2)}, we obtain the following.
\begin{theorem}{\label{Theorem: 2K2-free characterization}}
	Let $G$ be a graph that is non-isomorphic to any graph in Figure \ref{Figure: the graphs in EC(2K2)}. The graph $G$ is $2K_{2}$-free if and only if any $G$-contraction is $2K_{2}$-free.
\end{theorem}

\section{The \texorpdfstring{$P_{4}$}{P4}-Free Graphs}
By Proposition \ref{pro: no free-split for paths}, the following result follows.
\begin{corollary}{\label{Corollary: FCC(P4)}}
	There is no $P_{4}$-free-split graph.
\end{corollary}

\begin{proposition}{\label{Proposition: EC(P4)}}
	The graphs in Figure \ref{Figure: the graphs in EC(P4)} are the only critically $P_{4}$-exist graphs.
\end{proposition}

\begin{figure}[ht!]
	\centering
	\begin{tikzpicture}[hhh/.style={draw=black,circle,inner sep=2pt,minimum size=0.2cm}]
		\begin{scope}[shift={(0,0)}]
			\node 		(h) at (0,0)	 	{$H_{1}$};
			\begin{scope}[shift={(0,0.5)}]
				\node[hhh] (r) at (-1,0) 	{};
				\node[hhh] (s) at (-0.5,0) 	{};
				\node[hhh] (t) at (0,0) 	{};
				\node[hhh] (u) at (0.5,0) 	{};
				\node[hhh] (v) at (-0.25,1) 	{};
				
				\draw (r)--(s)--(t)--(u)--(v)--(r);
			\end{scope}
		\end{scope}
		
		\begin{scope}[shift={(3,0)}]
			\node 		(h) at (0,0)	 	{$H_{2}$};
			\begin{scope}[shift={(0,0.5)}]
				\node[hhh] (r) at (-1,0) 	{};
				\node[hhh] (s) at (-0.5,0) 	{};
				\node[hhh] (t) at (0,0) 	{};
				\node[hhh] (u) at (0.5,0) 	{};
				\node[hhh] (v) at (-0.5,1) 	{};
				
				\draw (r)--(s)--(t)--(u)--(v)--(r) (v)--(s);
			\end{scope}
		\end{scope}
		
		\begin{scope}[shift={(6,0)}]
			\node 		(h) at (0,0)	 	{$H_{3}$};
			\begin{scope}[shift={(0,0.5)}]
				\node[hhh] (r) at (-1,0) 	{};
				\node[hhh] (s) at (-0.5,0) 	{};
				\node[hhh] (t) at (0,0) 	{};
				\node[hhh] (u) at (0.5,0) 	{};
				\node[hhh] (v) at (-0.5,1) 	{};
				\node[hhh] (w) at (0,1) {};
				
				\draw (r)--(s)--(t)--(u)--(v)--(r) (v)--(s) (r)--(w)--(t) (w)--(u);
			\end{scope}
			
		\end{scope}
		
		\begin{scope}[shift={(1,-3)}]
			\node 		(h) at (0,0)	 	{$H_{4}$};
			\begin{scope}[shift={(0.5,0.5)}]
				\node[hhh] (r) at (-2,0) 	{};
				\node[hhh] (s) at (-1,0) 	{};
				\node[hhh] (t) at (0,0) 	{};
				\node[hhh] (u) at (1,0) 	{};
				\node[hhh] (v) at (-1,2) 	{};
				\node[hhh] (w) at (-1,0.8) 	{};
				\node (vdots) at (-1,1.4) 	{$\vdots$};	
				
				\draw (t)--(v)--(r)--(s)--(t)--(u) (r)--(w)--(t);
			\end{scope}
			
		\end{scope}
		
		\begin{scope}[shift={(5,-3)}]
			\node 		(h) at (0,0)	 	{$H_{5}$};
			\begin{scope}[shift={(0.5,0.5)}]
				\node[hhh] (r) at (-2,0) 	{};
				\node[hhh] (s) at (-1,0) 	{};
				\node[hhh] (t) at (0,0) 	{};
				\node[hhh] (u) at (1,0) 	{};
				\node[hhh] (v) at (-0.5,2) 	{};
				\node[hhh] (w) at (-0.5,0.8) 	{};
				\node (vdots) at (-0.5,1.4) 	{$\vdots$};	
				
				\draw (s)--(v)--(r)--(s)--(t)--(u)--(v)--(t) (r)--(w)--(s) (t)--(w)--(u);
			\end{scope}
			
		\end{scope}
	\end{tikzpicture}	
	\caption{The critically $P_{4}$-exist graphs}
	\label{Figure: the graphs in EC(P4)}
\end{figure}
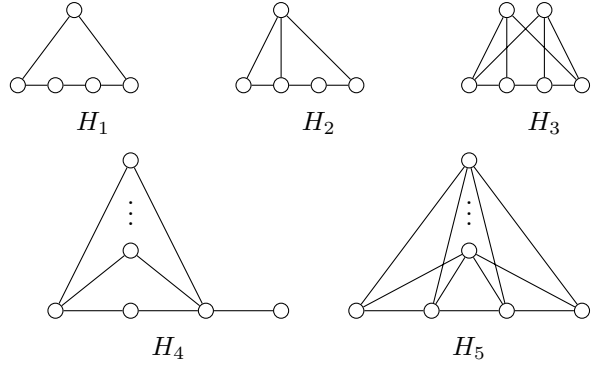

\begin{proof}
	\renewcommand{\qedsymbol}{\claimqed}
	Through this proof, we assume that $G$ is a critically $P_{4}$-exist graph with $S=\{r,s,t,u\}$ such that $G[S]$ is isomorphic to $P_{4}$ where $rs,st,tu \in E(G)$. By Theorem \ref{Theorem: critical H-exist: not S is independent}, $V(G) - S$ is independent. Thus, any vertex in $V(G) - S$ is adjacent to vertices only in $S$. By Corollary \ref{coro: no vertex adajcet to 1 or 2 or 3 or max degree in critical graph}, if $v \in V(G) - S$, then $v$ is nonadjacent to exactly: one vertex, two adjacent vertices, or three vertices inducing a path.

	Let $v,w \in V(G) - S$ such that $v$ is adjacent to exactly the leafs of $S$.
	If $w$ is adjacent to $r$ (or $u$), then in $G/rw$, $f(\{s,t,u,v\})$ induces $P_{4}$, which contradicts the fact that $G$ is a critically $P_{4}$-exist. If $w$ is adjacent to $s$ (or $t$), then in $G/sw$, $f(\{r,t,u,v\})$ induces $P_{4}$, which contradicts the fact that $G$ is a critically $P_{4}$-exist. Thus, if there is a vertex $v \in V(G) - S$ such that $v$ is adjacent to exactly the leaves of the path induced by $S$, then $G$ is isomorphic to $H_{1}$.

	Let $v \in V(G) - S$ such that $|N(v)|=2$.
	If $v$ is adjacent to exactly the leaves of the path induced by $S$, then $G$ is isomorphic to $H_{1}$. Assume that $v$ is nonadjacent to exactly the leaves of the path induced by $S$. Hence, $v$ is exactly adjacent to $\{r,t\}$ (or $\{s,u\}$). Assume that there is a vertex $w \in V(G) - S$ that is not exactly adjacent to $\{r,t\}$. Hence, $w$ is nonadjacent to exactly $\{r,u\}$. If $w$ is adjacent to $s$, then in $G/sw$, $f(\{r,t,u,v\})$ induces $P_{4}$, which contradicts the fact that $G$ is a critically $P_{4}$-exist. Hence, $w$ is adjacent to exactly $\{r,t,u\}$, however, in $G/st$, $f(\{r,u,v,w\})$ induces $P_{4}$, which contradicts the fact that $G$ is a critically $P_{4}$-exist. Thus, if $v \in V(G) - S$ such that $|N(v)|=2$, then $G$ is isomorphic to either $H_{1}$ or a graph in $H_{4}$.

	Let $v,w \in V(G) - S$ such that $|N(v)| = |N(w)|= 3$.
	Assume that $N(v) = N(w)$. W.l.o.g, let $N(v) = \{r,s,u\}$. In $G/sv$, $f(\{r,t,u,w\})$ induces $P_{4}$, which contradicts the fact that $G$ is a critically $P_{4}$-exist. Thus, if $v,w \in V(G) - S$ such that $|N(v)| = |N(w)|= 3$, then $N(v) \not= N(w)$.

	Let $v,w \in V(G) - S$ such that $|N(v)| = 3$.
	W.l.o.g, assume that $N(v) = \{r,s,u\}$. Assume that $N(w) =\{r,s,t,u\}$, however, in $G/sw$, $f(\{r,t,u,v\})$ induces $P_{4}$, which contradicts the fact that $G$ is a critically $P_{4}$-exist. Thus, if $v,w \in V(G) - S$ such that $|N(v)| = 3$, then $|N(w)| \not= 4$.

	Consequently, the possible critically $P_{4}$-exist graphs are those presented in Figure \ref{Figure: the graphs in EC(P4)} whose proofs, of being critically $P_{4}$-exist, are straightforward, which completes the proof.
	\renewcommand{\qedsymbol}{$\square$}
\end{proof}

By Theorem \ref{Theorem: characeterization}, Corollary \ref{Corollary: FCC(P4)}, and Proposition \ref{Proposition: EC(P4)}, we obtain the following.
\begin{theorem}{\label{Theorem: P4-free characterization}}
	Let $G$ be a graph that is non-isomorphic to any graph in Figure \ref{Figure: the graphs in EC(P4)}. The graph $G$ is $P_{4}$-free if and only if any $G$-contraction is $P_{4}$-free.
\end{theorem}

\section{The \texorpdfstring{$C_{4}$}{C4}-Free Graphs}
\begin{proposition}{\label{Proposition: CC(C4)}}
	The graphs in Figure \ref{Figure: the graphs in CC(C4)} are the only $C_{4}$-split graphs.
\end{proposition}
\begin{corollary}{\label{Corollary: FCC(C4)}}
	$C_{5}$ is the only $C_{4}$-free-split graph.
\end{corollary}
\begin{figure}[ht!]
	\centering
	\begin{tikzpicture}[hhh/.style={draw=black,circle,inner sep=2pt,minimum size=0.2cm},scale=1.5]
		\def \n {5}
		\def \radius {0.5cm}
		\def \margin {8} % margin in angles, depends on the radius
		
		\begin{scope}[shift={(0,0)}]
			\node 		(h) at (0,-0.8)	 	{$H_{1}$};
			\foreach \s in {1,...,\n}
			{
				\node[hhh] at ({360/\n * (\s - 1)}:\radius) {};
				\draw[ >=latex]  ({360/\n * (\s - 1)+\margin}:\radius)
				arc ({360/\n * (\s - 1)+\margin}:{360/\n * (\s)-\margin}:\radius);
			}
		\end{scope}
		
		\begin{scope}[shift={(1.5,0)}]
			\node 		(h) at (0,-0.8)	 	{$H_{2}$};
			\foreach \s in {1,...,\n}
			{
				\node[hhh] at ({360/\n * (\s - 1)}:\radius) {};
				\draw[ >=latex]  ({360/\n * (\s - 1)+\margin}:\radius)
				arc ({360/\n * (\s - 1)+\margin}:{360/\n * (\s)-\margin}:\radius);
			}
			\draw (0:\radius-2) --(360*2/\n:\radius-2) (360/\n:\radius-2) --(360*3/\n:\radius-2);
		\end{scope}
		
		\begin{scope}[shift={(3,0)}]
			\node 		(h) at (0,-0.8)	 	{$H_{3}$};
			\foreach \s in {1,...,\n}
			{
				\node[hhh] at ({360/\n * (\s - 1)}:\radius) {};
				\draw[ >=latex]  ({360/\n * (\s - 1)+\margin}:\radius)
				arc ({360/\n * (\s - 1)+\margin}:{360/\n * (\s)-\margin}:\radius);
			}
			\draw (0:\radius-2) --(360*2/\n:\radius-2);
		\end{scope}
		
		\begin{scope}[shift={(4.5,0)}]
			\node 		(h) at (0,-0.8)	 	{$H_{4}$};
			\foreach \s in {1,...,\n}
			{
				\node[hhh] at ({360/\n * (\s - 1)}:\radius) {};
				\draw[ >=latex]  ({360/\n * (\s - 1)+\margin}:\radius)
				arc ({360/\n * (\s - 1)+\margin}:{360/\n * (\s)-\margin}:\radius);
			}
			\draw (0:\radius-2) --(360*2/\n:\radius-2);
			\draw[ white, very thick,>=latex]  ({0+\margin}:\radius)
			arc ({0+\margin}:{360/\n * (1)-\margin}:\radius);
		\end{scope}
	\end{tikzpicture}	
	\caption{$C_{4}$-split graphs}
	\label{Figure: the graphs in CC(C4)}
\end{figure}
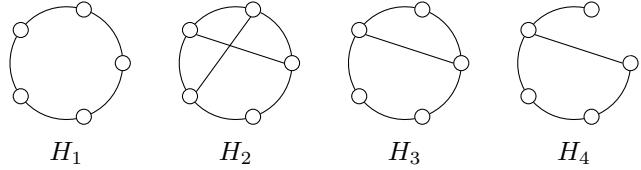

\begin{proposition}{\label{Proposition: EC(C4)}}
	The graphs in Figure \ref{Figure: the graphs in EC(C4)} are the only critically $C_{4}$-exist graphs.
\end{proposition}

\begin{figure}[ht!]
	\centering
	\begin{tikzpicture}[hhh/.style={draw=black,circle,inner sep=2pt,minimum size=0.2cm}]
		\begin{scope}[shift={(-3,0)}]
			\node 		(h) at (0,0)	 	{$H_{1}$};
			\begin{scope}[shift={(0,1.3)}]
				\def \n {4}
				\def \radius {0.7cm}
				\def \radiusCorrect {3}
				\def \margin {10} % margin in angles, depends on the radius
				\def \rotate {0}	% to rotate the cycle
				
				\node[hhh]  (e) at (-1.5,0)		{};
				\node[hhh]  (f) at (-2.5,0)		{};
				\node[text centered]  (g) at (-1.9,0)		{$\dots$};
				\draw ({360/\n * (2 - 1)+\rotate}:\radius)--(e)--({360/\n * (4 - 1)+\rotate}:\radius)
				({360/\n * (2 - 1)+\rotate}:\radius)--(f)--({360/\n * (4 - 1)+\rotate}:\radius);
				\foreach \s in {1,...,\n}
				{
					\node[hhh,fill=white] at ({360/\n * (\s - 1)+\rotate}:\radius) {};
					\draw[ >=latex]  ({360/\n * (\s - 1)+\margin + \rotate}:\radius)
					arc ({360/\n * (\s - 1)+\margin+\rotate}:{360/\n * (\s)-\margin + \rotate}:\radius);
				}
			\end{scope}
		\end{scope}
		
		\begin{scope}[shift={(0,0)}]
			\node 		(h) at (0,0)	 	{$H_{2}$};
			\begin{scope}[shift={(0,1.3)}]
				\def \n {4}
				\def \radius {0.7cm}
				\def \radiusCorrect {3}
				\def \margin {10} % margin in angles, depends on the radius
				\def \rotate {45}	% to rotate the cycle
				
				\foreach \s in {1,...,\n}
				{
					\node[hhh] at ({360/\n * (\s - 1)+\rotate}:\radius) {};
					\draw[ >=latex]  ({360/\n * (\s - 1)+\margin + \rotate}:\radius)
					arc ({360/\n * (\s - 1)+\margin+\rotate}:{360/\n * (\s)-\margin + \rotate}:\radius);
				}
				\node[hhh]  (e) at (0,0)		{};
				\draw ({360/\n * (1 - 1)+\rotate}:\radius-\radiusCorrect)--(e)--({360/\n * (2 - 1)+\rotate}:\radius-\radiusCorrect)
				({360/\n * (3 - 1)+\rotate}:\radius-\radiusCorrect)--(e)--({360/\n * (4 - 1)+\rotate}:\radius-\radiusCorrect);
			\end{scope}
		\end{scope}
		
		\begin{scope}[shift={(3,0)}]
			\node 		(h) at (0,0)	 	{$H_{3}$};
			\begin{scope}[shift={(0,1.3)}]
				\def \n {6}
				\def \radius {0.7cm}
				\def \radiusCorrect {3}
				\def \margin {10} % margin in angles, depends on the radius
				\def \rotate {0}	% to rotate the cycle
				
				\foreach \s in {1,...,\n}
				{
					\node[hhh] at ({360/\n * (\s - 1)+\rotate}:\radius) {};
					\draw[ >=latex]  ({360/\n * (\s - 1)+\margin + \rotate}:\radius)
					arc ({360/\n * (\s - 1)+\margin+\rotate}:{360/\n * (\s)-\margin + \rotate}:\radius);
				}
				\draw ({360/\n * (3 - 1)+\rotate}:\radius-\radiusCorrect)--({360/\n * (0)+\rotate}:\radius-\radiusCorrect)
				({360/\n * (5 - 1)+\rotate}:\radius-\radiusCorrect)--({360/\n * (0)+\rotate}:\radius-\radiusCorrect)
				({360/\n * (2 - 1)+\rotate}:\radius-\radiusCorrect)--({360/\n * (4-1)+\rotate}:\radius-\radiusCorrect)
				({360/\n * (6 - 1)+\rotate}:\radius-\radiusCorrect)--({360/\n * (4-1)+\rotate}:\radius-\radiusCorrect)
				({360/\n * (2 - 1)+\rotate}:\radius-\radiusCorrect)--({360/\n * (6-1)+\rotate}:\radius-\radiusCorrect)
				({360/\n * (3 - 1)+\rotate}:\radius-\radiusCorrect)--({360/\n * (5-1)+\rotate}:\radius-\radiusCorrect);
			\end{scope}
		\end{scope}
	\end{tikzpicture}	
	\caption{Critically $C_{4}$-exist graphs}
	\label{Figure: the graphs in EC(C4)}
\end{figure}
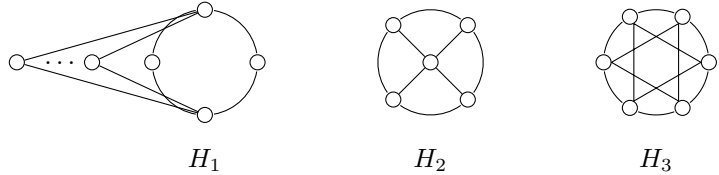

\begin{proof}
	\renewcommand{\qedsymbol}{\claimqed}
	Through this proof, we assume that $G$ is a critically $C_{4}$-exist graph with $S=\{r,s,t,u\}$ such that $G[S]$ is isomorphic to $C_{4}$ where $r$ and $t$ are adjacent to both $s$ and $u$. By Theorem \ref{Theorem: critical H-exist: not S is independent}, we note that $V(G) - S$ is independent. Thus, any vertex in $V(G) - S$ is adjacent to vertices only in $S$. By Corollary \ref{coro: no vertex adajcet to 1 or 2 or 3 or max degree in critical graph}, if $v \in V(G) - S$, then $v$ is nonadjacent to exactly: one vertex, two adjacent vertices, or three vertices.

	Let $v \in V(G) - S$ such that $|N(v)| =2$.	W.l.o.g, assume that $N(v) =\{r,t\}$. Let $w \in V(G)-S$, however, if $w$ is adjacent to $s$ (or $u$), then in $G/sw$, $f(\{r,t,u,v\})$ induces $C_{4}$, which contradicts the fact that $G$ is a critically $C_{4}$-exist. Thus, if $v \in V(G) - S$ such that $|N(v)| =2$, then $G$ is isomorphic to a graph in $H_{1}$.
	
	Let $v,w,x \in V(G) - S$ such that $|N(v)| = |N(w)| = 4$.
	Hence, $|N(x)| = 4$, however, in $G/rv$, $f(\{s,u,w,x\})$ induces $C_{4}$, which contradicts the fact that $G$ is a critically $C_{4}$-exist. Thus, if there is a vertex outside $S$ that is adjacent to every vertex in $S$ then $G$ is isomorphic to either $H_{2}$ or $H_{3}$.
	% $v,w \in V(G) - S$ such that $|N(v)| =|N(w)|=4$, then $G$ is isomorphic to $H_{3}$.
	
	Consequently, the possible critically $C_{4}$-exist graphs are those presented in Figure \ref{Figure: the graphs in EC(C4)} whose proofs, of being critically $C_{4}$-exist, are straightforward, which completes the proof.
	\renewcommand{\qedsymbol}{$\square$}
\end{proof}

By Theorem \ref{Theorem: characeterization}, Corollary \ref{Corollary: FCC(C4)}, and Proposition \ref{Proposition: EC(C4)}, we obtain the following.
\begin{theorem}{\label{Theorem: C4-free characterization}}
	Let $G$ be a $C_{5}$-free graph that is non-isomorphic to any graph in Figure \ref{Figure: the graphs in EC(C4)}. The graph $G$ is $C_{4}$-free if and only if any $G$-contraction is $C_{4}$-free.
\end{theorem}

\section{The \texorpdfstring{$C_{5}$}{C5}-Free Graphs}
\begin{proposition}{\label{Proposition: CC(C5)}}
	The graphs in Figure \ref{Figure: the graphs in CC(C5)} are the only $C_{5}$-split graphs.
\end{proposition}

\begin{figure}[!ht]
	\centering
	\begin{tikzpicture}[hhh/.style={draw=black,circle,inner sep=1.5pt,minimum size=0.15cm},scale=1.3]
		\def \n {6}
		\def \radius {0.4cm}
		\def \margin {8} % margin in angles, depends on the radius
		
		\begin{scope}[shift={(0,0)}]
			\node 		(h) at (0,-0.7)	 	{$H_{1}$};
			\foreach \s in {1,...,\n}
			{
				\node[hhh] at ({360/\n * (\s - 1)}:\radius) {};
				\draw[ >=latex]  ({360/\n * (\s - 1)+\margin}:\radius)
				arc ({360/\n * (\s - 1)+\margin}:{360/\n * (\s)-\margin}:\radius);
			}
		\end{scope}
		
		\begin{scope}[shift={(1.5,0)}]
			\node 		(h) at (0,-0.7)	 	{$H_{2}$};
			\foreach \s in {1,...,\n}
			{
				\node[hhh] at ({360/\n * (\s - 1)}:\radius) {};
				\draw[ >=latex]  ({360/\n * (\s - 1)+\margin}:\radius)
				arc ({360/\n * (\s - 1)+\margin}:{360/\n * (\s)-\margin}:\radius);
			}
			\draw (0:\radius-2) --(360*2/\n:\radius-2) (360/\n:\radius-2) --(360*3/\n:\radius-2);
		\end{scope}
		
		\begin{scope}[shift={(3,0)}]
			\node 		(h) at (0,-0.7)	 	{$H_{3}$};
			\foreach \s in {1,...,\n}
			{
				\node[hhh] at ({360/\n * (\s - 1)}:\radius) {};
				\draw[ >=latex]  ({360/\n * (\s - 1)+\margin}:\radius)
				arc ({360/\n * (\s - 1)+\margin}:{360/\n * (\s)-\margin}:\radius);
			}
			\draw (0:\radius-2) --(360*2/\n:\radius-2);
		\end{scope}
		
		\begin{scope}[shift={(4.5,0)}]
			\node 		(h) at (0,-0.7)	 	{$H_{4}$};
			\foreach \s in {1,...,\n}
			{
				\node[hhh] at ({360/\n * (\s - 1)}:\radius) {};
				\draw[ >=latex]  ({360/\n * (\s - 1)+\margin}:\radius)
				arc ({360/\n * (\s - 1)+\margin}:{360/\n * (\s)-\margin}:\radius);
			}
			\draw (0:\radius-2) --(360*2/\n:\radius-2);
			\draw[ white, very thick,>=latex]  ({0+\margin}:\radius)
			arc ({0+\margin}:{360/\n * (1)-\margin}:\radius);
		\end{scope}
	\end{tikzpicture}	
	\caption{$C_{5}$-split graphs}
	\label{Figure: the graphs in CC(C5)}
\end{figure}
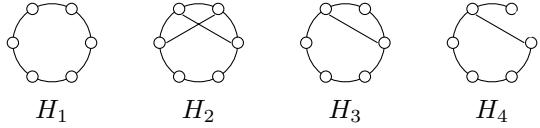

\begin{corollary}{\label{Corollary: FCC(C5)}}
	$C_{6}$ is the only $C_{5}$-free-split graph.
\end{corollary}

\begin{proposition}{\label{Proposition: EC(C5)}}
	The graphs in Figure \ref{Figure: the graphs in EC(C5)} are the only critically $C_{5}$-exist graphs. 
\end{proposition}
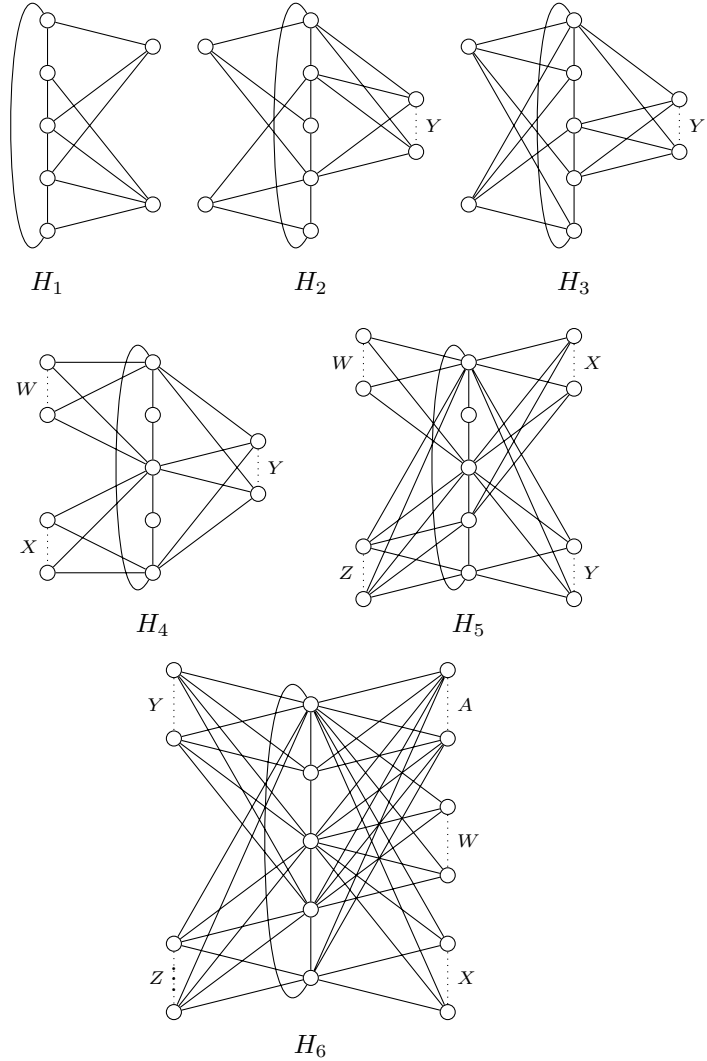
\begin{figure}[!ht]
	\centering
	\begin{tikzpicture}[hhh/.style={draw=black,circle,inner sep=2pt,minimum size=0.2cm},scale=0.7]
		\begin{scope}[shift={(0,0)}]
			\node 	   (label) at (0,0) 		{$H_{1}$};
			\node[hhh] (v) at (0,1) 	{};
			\node[hhh] (u) at (0,2) 	{};
			\node[hhh] (t) at (0,3) 	{};
			\node[hhh] (s) at (0,4) 	{};
			\node[hhh] (r) at (0,5) 	{};
			
			\node[hhh] (v3) at (2,4.5) 	{};
			\node[hhh] (v4) at (2,1.5) 	{};

			\draw (r) -- (s) -- (t) -- (u) -- (v)
			
			(r) --(v3) -- (t) (u) --(v3)
			(s) --(v4) -- (t) (u) --(v4) --(v)
			
			(r) to[in=-120,out=120]  (v);
		\end{scope}

		\begin{scope}[shift={(5,0)}]
			\node 	   (label) at (0,0) 		{$H_{2}$};
			\node[hhh] (r) at (0,1) 	{};
			\node[hhh] (s) at (0,2) 	{};
			\node[hhh] (t) at (0,3) 	{};
			\node[hhh] (u) at (0,4) 	{};
			\node[hhh] (v) at (0,5) 	{};
			
			\node[hhh] (v41) at (-2,4.5) 	{};
			\node[hhh] (v42) at (-2,1.5) 	{};
			
			\node[hhh] (v31) at (2,3.5) 	{};
			\node[hhh] (v32) at (2,2.5) 	{};
			%\node (v3dots) at (2,3.1) 	{$\vdots$};	
			\draw[dotted] (v31) to node[right]{\scriptsize $Y$} (v32);		
			
			\draw (r) -- (s) -- (t) -- (u) -- (v)
			
			(v) --(v41) -- (t) (s) --(v41)
			(u) --(v42) -- (s) (r) --(v42)
			
			(s) --(v31) -- (v) (u) --(v31)
			(s) --(v32) -- (v) (u) --(v32)
			
			(v) to[in=-120,out=120]  (r);
		\end{scope}

		\begin{scope}[shift={(10,0)}]
			\node 	   (label) at (0,0) 		{$H_{3}$};
			\node[hhh] (r) at (0,1) 	{};
			\node[hhh] (s) at (0,2) 	{};
			\node[hhh] (t) at (0,3) 	{};
			\node[hhh] (u) at (0,4) 	{};
			\node[hhh] (v) at (0,5) 	{};
			
			\node[hhh] (v41) at (-2,4.5) 	{};
			\node[hhh] (v42) at (-2,1.5) 	{};
			
			\node[hhh] (v31) at (2,3.5) 	{};
			\node[hhh] (v32) at (2,2.5) 	{};
			%\node (v3dots) at (2,3.1) 	{$\vdots$};	
			\draw[dotted] (v31) to node[right]{\scriptsize $Y$} (v32);			
			
			\draw (r) -- (s) -- (t) -- (u) -- (v)
			
			(s) --(v41) -- (r) (u) --(v41) -- (v) 
			(r) --(v42) -- (u) (v) --(v42) -- (t)
			
			(v) --(v31) -- (t) (s) --(v31)
			(v) --(v32) -- (t) (s) --(v32)
			
			(v) to[in=-120,out=120]  (r);
		\end{scope}
		
		\begin{scope}[shift={(2,-6.5)}]
			\node 	   (label) at (0,0) 		{$H_{4}$};
			\node[hhh] (r) at (0,1) 	{};
			\node[hhh] (s) at (0,2) 	{};
			\node[hhh] (t) at (0,3) 	{};
			\node[hhh] (u) at (0,4) 	{};
			\node[hhh] (v) at (0,5) 	{};
			
			\node[hhh] (v211) at (-2,5) 	{};
			\node[hhh] (v212) at (-2,4) 	{};
			%	\node (v2dots) at (-2,4.6) 	{$\vdots$};
			\draw[dotted] (v211) to node[left]{\scriptsize $W$} (v212);
			
			\node[hhh] (v221) at (-2,2) 	{};
			\node[hhh] (v222) at (-2,1) 	{};
			%	\node (v2dots) at (-2,1.6) 	{$\vdots$};
			\draw[dotted] (v221) to node[left]{\scriptsize $X$} (v222);
			
			\node[hhh] (v31) at (2,3.5) 	{};
			\node[hhh] (v32) at (2,2.5) 	{};
			%	\node (v3dots) at (2,3.1) 	{$\vdots$};
			\draw[dotted] (v31) to node[right]{\scriptsize $Y$} (v32);			
			
			\draw (r) -- (s) -- (t) -- (u) -- (v)
			
			(v) --(v211) -- (t) (v) --(v212) -- (t) 
			(t) --(v221) -- (r) (t) --(v222) -- (r)
			
			(t) --(v31) -- (r) (v) --(v31)
			(t) --(v32) -- (r) (v) --(v32)
			
			(v) to[in=-120,out=120]  (r);
		\end{scope}

		\begin{scope}[shift={(8,-6.5)}]
			\node 	   (label) at (0,0) 		{$H_{5}$};
			\node[hhh] (r) at (0,1) 	{};
			\node[hhh] (s) at (0,2) 	{};
			\node[hhh] (t) at (0,3) 	{};
			\node[hhh] (u) at (0,4) 	{};
			\node[hhh] (v) at (0,5) 	{};
			
			\node[hhh] (v21) at (-2,5.5) 	{};
			\node[hhh] (v22) at (-2,4.5) 	{};
			%	\node (v2dots) at (-2,5.1) 	{$\vdots$};
			\draw[dotted] (v21) to node[left]{\scriptsize $W$} (v22);
			
			\node[hhh] (v421) at (-2,1.5) 	{};
			\node[hhh] (v422) at (-2,0.5) 	{};
			%	\node (v4dots) at (-2,1.1) 	{$\vdots$};
			\draw[dotted] (v421) to node[left]{\scriptsize $Z$} (v422);
			
			\node[hhh] (v311) at (2,5.5) 	{};
			\node[hhh] (v312) at (2,4.5) 	{};
			%	\node (v3dots) at (2,5.1) 	{$\vdots$};	
			\draw[dotted] (v311) to node[right]{\scriptsize $X$} (v312);
			
			\node[hhh] (v321) at (2,0.5) 	{};
			\node[hhh] (v322) at (2,1.5) 	{};
			%	\node (v3dots) at (2,1.1) 	{$\vdots$};	
			\draw[dotted] (v321) to node[right]{\scriptsize $Y$} (v322);		
			
			\draw (r) -- (s) -- (t) -- (u) -- (v)
			
			(v) --(v21) -- (t) 
			(v) --(v22) -- (t) 
			
			(t) --(v421) -- (s) (r) --(v421) -- (v) 
			(t) --(v422) -- (s) (r) --(v422) -- (v)
			
			(v) --(v311) -- (t) (s) --(v311)
			(v) --(v312) -- (t) (s) --(v312)
			
			(t) --(v321) -- (r)	(v) --(v321)
			(t) --(v322) -- (r)	(v) --(v322)
			
			(v) to[in=-120,out=120]  (r);
		\end{scope}

		\begin{scope}[shift={(5,-14.5)},scale=1.3]
			\node 	   (label) at (0,0) 		{$H_{6}$};
			\node[hhh] (r) at (0,1) 	{};
			\node[hhh] (s) at (0,2) 	{};
			\node[hhh] (t) at (0,3) 	{};
			\node[hhh] (u) at (0,4) 	{};
			\node[hhh] (v) at (0,5) 	{};
			
			\node[hhh] (v51) at (2,5.5) 	{};
			\node[hhh] (v52) at (2,4.5) 	{};
			%	\node (v5dots) at (2,5.1) 	{$\vdots$};	
			\draw[dotted] (v51) to node[right]{\scriptsize $A$} (v52);

			\node[hhh] (v411) at (-2,5.5) 	{};
			\node[hhh] (v412) at (-2,4.5) 	{};
			%	\node (v4dots) at (-2,5.1) 	{$\vdots$};
			\draw[dotted] (v411) to node[left]{\scriptsize $Y$} (v412);
			
			\node[hhh] (v421) at (-2,1.5) 	{};
			\node[hhh] (v422) at (-2,0.5) 	{};
			\node (v4dots) at (-2,1.1) 	{$\vdots$};
			\draw[dotted] (v421) to node[left]{\scriptsize $Z$} (v422);
			
			\node[hhh] (v311) at (2,3.5) 	{};
			\node[hhh] (v312) at (2,2.5) 	{};
			%	\node (v3dots) at (2,3.1) 	{$\vdots$};	
			\draw[dotted] (v311) to node[right]{\scriptsize $W$} (v312);
			
			\node[hhh] (v321) at (2,0.5) 	{};
			\node[hhh] (v322) at (2,1.5) 	{};
			%	\node (v3dots) at (2,1.1) 	{$\vdots$};	
			\draw[dotted] (v321) to node[right]{\scriptsize $X$} (v322);		
			
			\draw (r) -- (s) -- (t) -- (u) -- (v)
			(r) --(v51) -- (s) (t) --(v51) -- (u) (v) -- (v51)
			(r) --(v52) -- (s) (t) --(v52) -- (u) (v) -- (v52)
			
			(v) --(v411) -- (u) (t) --(v411) -- (s) 
			(v) --(v412) -- (u) (t) --(v412) -- (s)
			
			(t) --(v421) -- (s) (r) --(v421) -- (v) 
			(t) --(v422) -- (s) (r) --(v422) -- (v)
			
			(v) --(v311) -- (t) (s) --(v311)
			(v) --(v312) -- (t) (s) --(v312)
			
			(t) --(v321) -- (r)	(v) --(v321)
			(t) --(v322) -- (r)	(v) --(v322)
			
			(v) to[in=-120,out=120]  (r);
		\end{scope}
		
	\end{tikzpicture}
	\caption{Critically $C_{5}$-exist graphs}
	\label{Figure: the graphs in EC(C5)}
\end{figure}

%Before we prove Proposition \ref{Proposition: EC(C5)}, we recall the following notion. In a path (tree), the vertices with degree one are called leaves.

\begin{proof}
	\renewcommand{\qedsymbol}{\claimqed}
	Through this proof, we assume that $G$ is a critically $C_{5}$-exist graph with $S=\{r,s,t,u,v\}$  such that $G[S]$ is isomorphic to a $C_{5}$ where $rs,st,tu,uv$, and $vr$ are the edges in $E(G[S])$. By Theorem \ref{Theorem: critical H-exist: not S is independent}, we note that $V(G) - S$ is independent. Thus, any vertex in $V(G) - S$ is adjacent to vertices only in $S$. By Corollary \ref{coro: no vertex adajcet to 1 or 2 or 3 or max degree in critical graph}, if $w \in V(G) - S$, then the neighborhood of $w$ is not exactly: one vertex, two adjacent vertices, or three vertices that induce a path.

	\begin{claim}{\label{Claim: EC(C5) neighborhood of lenght 2, then no one is adajcent to the vertex in between}}
		If $w,x \in V(G) - S$ such that $|N(w)| =2$, then for any vertex $y$ in $N(x)$, $N(w) \cup \{y\}$ do not induce $P_{3}$. 
	\end{claim}
	\begin{proof}
		For the sake of contradiction, and w.l.o.g., assume that there are vertices $w,x \in V(G) - S$ such that $N(w) = \{r,t\}$ and there is a vertex $y \in |N_{S}(x)|$, where $N(w) \cup \{y\}$ induces $P_{3}$. Consequently, $y=s$. In $G/sx$, the set $f(\{r,t,u,v,w\})$ induces $C_{5}$, which contradicts the fact that $G$ is a critically $C_{5}$-exist. 
	\end{proof}
	
	\begin{claim}{\label{Claim: EC(C5) no two v3 spanning the S}}
		If $w,x,y \in V(G) - S$ such that $|N(w)|=|N(x)|=3$ where $N(w) \cup N(x) = S$, then $y$ is nonadjacent to any vertex in $S$ that is adjacent to $N(w) \cap N(x)$.
	\end{claim}
	\begin{proof}
		W.l.o.g., assume that there are vertices $w,x,y \in V(G) - S$ such that $N(w) = \{r,t,u\}$ and $N(w) \cup N(x) = S$. As a result, $N(x) = \{s,t,v\}$ or $\{s,u,v\}$. Because of symmetry, we assume that $N(x) = \{s,t,v\}$. For the sake of contradiction, we assume that $y$ is adjacent to $s$ (or $u$). However, in $G/sy$ (or $G/uy$), the set $f(\{r,t,v,w,x\})$ induces $C_{5}$ which contradicts the fact that $G$ is a critically $C_{5}$-exist.
	\end{proof}
	
	\begin{claim}{\label{Claim: EC(C5) no two v4 spanning the S}}
		If $w,x,y \in V(G) - S$ such that $|N(w)| = |N(x)| = 4$, where $N(w) \cap N(x)$ induces $P_{3}$, then $y$ is nonadjacent to the leaves in the path induced by $N(w) \cap N(x)$.
	\end{claim}
	\begin{proof}
		W.l.o.g., assume that there are vertices $w,x,y \in V(G) - S$ such that $N(w) = \{r,s,t,u\}$, while $N(w) \cap N(x)$ induces $P_{3}$. As a result, $N(x)=\{s,t,u,v\}$ or $\{r,t,u,v\}$. Because of symmetry, we assume that $N(x)=\{s,t,u,v$ $\}$. For the sake of contradiction, we assume that $y$ is adjacent to $s$ (or $u$). However, in $G/sy$ (or $G/uy$), the set $f(\{r,t,v,w,x\})$ induces $C_{5}$, which contradicts the fact that $G$ is a critically $C_{5}$-exist.
	\end{proof}
	
	\begin{claim}{\label{Claim: EC(C5) no two v3 and v4}}
		If $w,x \in V(G) - S$ such that $|N(w)| = 3$ and $|N(x)| = 4$ where neither of the leaves of the $P_{4}$ induced by $N(x)$ are in $N(w)$, then $G$ is isomorphic to $H_{1}$.
	\end{claim}
	\begin{proof}
		W.l.o.g., assume vertices $w,x \in V(G) - S$ such that $N(w) = \{r,t,u\}$ and $N(x) = \{s,t,u,v\}$. For the sake of contradiction, we assume that there is a vertex $y \in V(G) - S$. If $y$ is adjacent to $s$ (or $u$), then in $G/sy$ (or $G/uy$), the set $f(\{r,t,v,w,x\})$ induces $C_{5}$, which contradicts the fact that $G$ is a critically $C_{5}$-exist. In contrast, if $y$ is adjacent to $t$ (or $v$), then in $G/ty$ (or $G/vy$), the set $f(\{r,s,u,w,x\})$ induces $C_{5}$, which contradicts the fact that $G$ is a critically $C_{5}$-exist.
	\end{proof}
	
	By Claims \ref{Claim: EC(C5) neighborhood of lenght 2, then no one is adajcent to the vertex in between} to \ref{Claim: EC(C5) no two v3 and v4}, we deduce that the possible critically $C_{5}$-exist graphs are the ones presented in Figure \ref{Figure: the graphs in EC(C5)}. To complete the proof, we demonstrate that every graph in Figure \ref{Figure: the graphs in EC(C5)} is a critically $C_{5}$-exist.
	\begin{claim}{\label{Claim: No C6}}
		There is no $C_{6}$-exist graph in Figure \ref{Figure: the graphs in EC(C5)}.
	\end{claim}
	\begin{proof}
		Assume for the sake of contradiction, there is a graph $G$ in in Figure \ref{Figure: the graphs in EC(C5)} that is $C_{6}$-exist. Moreover, $T \subseteq V(G)$ where $G[T]$ induces $C_{6}$. Let $S \subseteq V(G)$ such that $G[S]$ induces $C_{5}$. No more than two vertices in $S$ can form an independent set. Hence $T$ can contain at most two vertices from $S$. Consequently, $T$ has four vertices from $V(G)-S$, however, two of such four vertices are adjacent, which contradicts the fact that $V(G)-S$ is independent set.
	\end{proof}
	All graphs in Figure \ref{Figure: the graphs in EC(C5)} are $C_6$-free. Consequently, any contraction of a graph of Figure \ref{Figure: the graphs in EC(C5)} has an induced $C_{5}$, then this $C_{5}$ must be an induced subgraph of the original graph too. Thus, it would be sufficient to prove that in every graph $G$ in Figure \ref{Figure: the graphs in EC(C5)}, the edges of $G$ are critical for every induced $C_{5}$ in $G$.
	
	\begin{claim}{\label{Claim: EC(C5) H1 is a critical C5-exist}}
		The graph $H_{1}$ in Figure \ref{Figure: the graphs in EC(C5)} is a critically $C_{5}$-exist.
	\end{claim}
	\begin{proof}
		The graph $H_{1}$ in Figure \ref{Figure: the graphs in EC(C5)} is isomorphic to a graph $G$ that contains a vertex subset $S=\{r,s,t,u,v\}$ that induces $C_{5}$, where $rs,st,tu,uv,vr \in E(G)$. Moreover, $V(G)= S \cup \{w,x\}$ such that $N(w)=\{r,t,u\}$ and $N(x)=\{s,t,u,v\}$. Clearly, the vertex subsets of $V(G)$ that induce $C_{5}$ are $S$, $\{r,s,u,w,x\}$, and $\{r,t,v,w,x\}$. Furthermore, it is straightforward that every edge in $E(G)$ is $C_{5}$-critical for the previous three vertex subsets. Thus, $G$ is a critically $C_{5}$-exist.
	\end{proof}

	\begin{claim}{\label{Claim: EC(C5) H2 is a critical C5-exist}}
		The graph $H_{2}$ in Figure \ref{Figure: the graphs in EC(C5)} is a critically $C_{5}$-exist.
	\end{claim}
	\begin{proof}
		The graph $H_{2}$ in Figure \ref{Figure: the graphs in EC(C5)} is isomorphic to a graph $G$ that contains a vertex subset $S=\{r,s,t,u,v\}$ that induces $C_{5}$, where $rs,st,tu,uv,vr \in E(G)$. Moreover, $V(G)= S \cup \{w,x\} \cup Y$, such that $N(w)= \{r,t,u\}$, $N(x)= \{s,u,v\}$, $N(y \in Y) =\{r,s,u\}$, and $|Y| \geq 0$.
		
		We prove that $G$ contains only two induced $C_{5}$. For any $i,j$ where $1 \leq i < j \leq |Y|$, if both $y_{i}$ and $y_{j}$ are in a vertex subset that induces a cycle in $G$, then this vertex subset would be $\{r,u,y_{i},y_{j}\}$ and $\{s,u,y_{i},y_{j}\}$. Thus, no induced $C_{5}$ in $G$ contains both $y_{i}$ and $y_{j}$. Moreover, the vertex subsets that induce a cycle in $G$ that contain $y_{i}$ but not $y_{j}$ are $\{r,s,y_{i}\}$, $\{r,u,v,y_{i}\}$, $\{r,u,y_{i},w\}$, $\{s,t,u,y_{i}\}$, and $\{s,u,x,y_{i}\}$. Clearly, none of them induces $C_{5}$. Thus, no induced $C_{5}$ in $G$ contain a vertex $y_{i}$ for any $i \leq l$. Furthermore, no four vertices from $S$ with either $w$ or $x$ induce a $C_{5}$. Consequently, no vertex subset that induces $C_{5}$ in $G$ contains either $w$ or $x$. As a result, a cycle $C_{5}$ is induced in $G$ only by $S$ or $\{r,s,u,w,x\}$. We note that any edge in $E(G)$ is a critical edge for all subsets that induce $C_{5}$. Thus, $G$ is a critically $C_{5}$-exist.
	\end{proof}
	
	The proof of the following claim can be performed in a similar way as that one of Claim \ref{Claim: EC(C5) H2 is a critical C5-exist}. We will just explain the structure of the graph $H_3$, the remaining part is left to the interested reader.
	\begin{claim}{\label{Claim: EC(C5) H3 is a critical C5-exist}}
		The graph $H_{3}$ in Figure \ref{Figure: the graphs in EC(C5)} is a critically $C_{5}$-exist.
	\end{claim}
	\begin{proof}
		The graph $H_{3}$ in Figure \ref{Figure: the graphs in EC(C5)} is isomorphic to a graph $G$ that contains a vertex subset $S=\{r,s,t,u,v\}$ that induces $C_{5}$, where $rs,st,tu,uv,vr \in E(G)$. Moreover, $V(G)= S \cup \{w,x\} \cup Y$, such that $N(w)= \{r,s,t,u\}$, $N(x)= \{s,t,u,v\}$, $N(y \in Y) =\{r,t,v\}$, and $|Y| \geq 0$.
	\end{proof}
	
	\begin{claim}{\label{Claim: EC(C5) H4 is a critical C5-exist}}
		The graph $H_{4}$ in Figure \ref{Figure: the graphs in EC(C5)} is a critically $C_{5}$-exist.
	\end{claim}
	\begin{proof}
		The graph $H_{4}$ in Figure \ref{Figure: the graphs in EC(C5)} is isomorphic to a graph $G$ that contains a vertex subset $S=\{r,s,t,u,v\}$ that induces $C_{5}$, where $rs,st,tu,uv,vr \in E(G)$. Moreover, $V(G)= S \cup W \cup X \cup Y$, such that $N(w \in W)= \{r,t\}$, $N(x \in X)= \{r,u\}$, $N(y \in Y) =\{r,t,u\}$, and $|W|,|X|,|Y| \geq 0$.
		
		For any $k,p$ where $1 \leq k < p \leq |Y|$, if both $y_{k}$ and $y_{p}$ are in a vertex subset that induces cycle in $G$, then this vertex subset would be $\{r,u,y_{k},y_{p}\}$ or $\{r,t,y_{k},y_{p}\}$. Thus, no induced $C_{5}$ in $G$ contains both $y_{k}$ and $y_{p}$. Additionally, the induced cycles in $G$ that contain $y_{k}$ but not $y_{p}$ (or vice versa) are of length less than five. Thus, no induced $C_{5}$ in $G$ contains one vertex $y_{k}$ for any $k \leq n$. Indeed, a vertex subset in $G$ that induces a $C_{5}$ is composed of $\{r,t,u\}$, one vertex from $W \cup \{s\}$, and one vertex from $X \cup \{v\}$. We conclude that any edge in $E(G)$ is a critical edge for all subsets that induce $C_{5}$. Thus, $G$ is a critically $C_{5}$-exist.
	\end{proof}

	\begin{claim}{\label{Claim: EC(C5) H5 is a critical C5-exist}}
		The graph $H_{5}$ in Figure \ref{Figure: the graphs in EC(C5)} is a critically $C_{5}$-exist.
	\end{claim}
	\begin{proof}
		The graph $H_{5}$ in Figure \ref{Figure: the graphs in EC(C5)} is isomorphic to a graph $G$ that contains a vertex subset $S=\{r,s,t,u,v\}$ that induces $C_{5}$, where $rs,st,tu,uv,vr \in E(G)$. Moreover, $V(G)= S \cup W \cup X \cup Y \cup Z$, such that $N(w \in W)= \{r,t\}$, $N(x \in X)= \{r,t,u\}$, $N(y \in Y) =\{r,t,v\}$, $N(z \in Z)=\{r,t,u,v\}$, and $|W|,|X|,|Y|,|Z| \geq 0$.
		
		It is clear that no vertex subset in $G$ that induces $C_{5}$ contains two vertices from $V(G) - S$. Moreover, we can prove that no vertex subset in $G$ that induces $C_{5}$ contains one vertices from $V(G) - (S \cup W)$. As a result, the vertex subsets that induces a $C_{5}$ in $G$ are either $S$ or $\{r,t,u,v\}$ together with one vertex from $\{w_{1},w_{2}, \dots , w_{l}\}$. Any edge in $E(G)$ is a critical for all subsets that induce $C_{5}$. Thus, $G$ is a critically $C_{5}$-exist.
	\end{proof}
	
	Being similar to the proof of Claim \ref{Claim: EC(C5) H5 is a critical C5-exist}, the proof of Claim \ref{Claim: EC(C5) H6 is a critical C5-exist} is left for the interested reader; however, we explain the structure of $H_{6}$ in Figure \ref{Figure: the graphs in EC(C5)} in the proof. 
	\begin{claim}{\label{Claim: EC(C5) H6 is a critical C5-exist}}
		The graph $H_{6}$ in Figure \ref{Figure: the graphs in EC(C5)} is a critically $C_{5}$-exist.
	\end{claim}
	\begin{proof}
		The graph $H_{6}$ in Figure \ref{Figure: the graphs in EC(C5)} is isomorphic to a graph $G$ that contains a vertex subset $S=\{r,s,t,u,v\}$ that induces $C_{5}$, where $rs,st,tu,uv,vr \in E(G)$. Moreover, $V(G)= S \cup W \cup X \cup Y \cup Z \cup A$, such $N(w \in W)=\{r,t,u\}$, $N(x \in X)=\{r,t,v\}$, $N(y \in Y)=\{r,s,t,u\}$, $N(z \in Z)=\{r,t,u,v\}$, $N(a\in A)=S$, and $|W|,|X|,|Y|,|Z|,|A| \geq 0$.
	\end{proof}
	
	By Claims \ref{Claim: No C6} to \ref{Claim: EC(C5) H6 is a critical C5-exist}, the proof is complete.
	\renewcommand{\qedsymbol}{$\square$}
\end{proof}

By Theorem \ref{Theorem: characeterization}, Corollary \ref{Corollary: FCC(C5)}, and Proposition \ref{Proposition: EC(C5)}, we obtain the following.
\begin{theorem}{\label{Theorem: C5-free characterization}}
	Let $G$ be a $C_{6}$-free graph that is non-isomorphic to any graph in Figure \ref{Figure: the graphs in EC(C5)}. The graph $G$ is $C_{5}$-free if and only if any $G$-contraction is $C_{5}$-free.
\end{theorem}

\section{Split Graphs}
Split graphs were introduced in \cite{hammer1977} and were characterized as follows:
\begin{theorem}\cite{hammer1977}\label{Theorem: split graph hammer charcaterization}
	A graph $G$ is split if and only if $G$ is $\{2K_{2}, C_{4}, C_{5}\}$-free.
\end{theorem}
Thus, we call a graph that is $\{2K_{2}, C_{4}, C_{5}\}$-exist \emph{non-split} graph.
Additionally, split graphs have been characterized in \cite{hammer1977} as chordal graphs whose complements are also chordal. Furthermore, it was characterized by its degree sequences in \cite{hammer1981splittance}. Moreover, further properties of split graphs are studied in \cite{bertossi1984dominating,kratsch1996toughness,merris2003split}.

By Theorems \ref{Theorem: characeterization}, \ref{Theorem: 2K2-free characterization}, \ref{Theorem: C4-free characterization}, and \ref{Theorem: C5-free characterization} and Propositions \ref{pro: only one free-split for cycles}, \ref{Lemma: Cycle contraction}, and \ref{Lemma: EC(C3)}, we obtain:

\begin{theorem}{\label{Theorem: split characterization short form}}
	Let $G$ be a graph that is non-isomorphic to any graph in Figure \ref{Figure: the graphs in EC(split)}. The graph $G$ is split if and only if any $G$-contraction is split.
\end{theorem}
The class of split graphs is a closed class under edge contraction. The definition of a split graph implies that by contraction of an arbitrary edge in a split graph leads to another split graph. So the contribution of Theorem \ref{Theorem: split characterization short form} is in listing the critically non-split graphs.

\begin{figure}[ht!]
	\centering
	\begin{tikzpicture}[hhh/.style={draw=black,circle,inner sep=2pt,minimum size=0.2cm}]
		\begin{scope}[shift={(-2,0)}]
			\node 		(h) at (0,0)	 	{$H_{1}$};
			\begin{scope}[shift={(0,1.3)}]
				\def \n {4}
				\def \radius {0.7cm}
				\def \radiusCorrect {3}
				\def \margin {10} % margin in angles, depends on the radius
				\def \rotate {0}	% to rotate the cycle
				
				\node[hhh]  (e) at (-1.5,0)		{};
				\node[hhh]  (f) at (-2.5,0)		{};
				\node[text centered]  (g) at (-1.9,0)		{$\dots$};
				\draw ({360/\n * (2 - 1)+\rotate}:\radius)--(e)--({360/\n * (4 - 1)+\rotate}:\radius)
				({360/\n * (2 - 1)+\rotate}:\radius)--(f)--({360/\n * (4 - 1)+\rotate}:\radius);
				\foreach \s in {1,...,\n}
				{
					\node[hhh,fill=white] at ({360/\n * (\s - 1)+\rotate}:\radius) {};
					\draw[ >=latex]  ({360/\n * (\s - 1)+\margin + \rotate}:\radius)
					arc ({360/\n * (\s - 1)+\margin+\rotate}:{360/\n * (\s)-\margin + \rotate}:\radius);
				}
			\end{scope}
		\end{scope}
		
		\begin{scope}[shift={(0.5,0)}]
			\node 		(h) at (0,0)	 	{$H_{2}$};
			\begin{scope}[shift={(0,1.3)}]
				\def \n {4}
				\def \radius {0.7cm}
				\def \radiusCorrect {3}
				\def \margin {10} % margin in angles, depends on the radius
				\def \rotate {45}	% to rotate the cycle
				
				\foreach \s in {1,...,\n}
				{
					\node[hhh] at ({360/\n * (\s - 1)+\rotate}:\radius) {};
					\draw[ >=latex]  ({360/\n * (\s - 1)+\margin + \rotate}:\radius)
					arc ({360/\n * (\s - 1)+\margin+\rotate}:{360/\n * (\s)-\margin + \rotate}:\radius);
				}
				\node[hhh]  (e) at (0,0)		{};
				\draw ({360/\n * (1 - 1)+\rotate}:\radius-\radiusCorrect)--(e)--({360/\n * (2 - 1)+\rotate}:\radius-\radiusCorrect)
				({360/\n * (3 - 1)+\rotate}:\radius-\radiusCorrect)--(e)--({360/\n * (4 - 1)+\rotate}:\radius-\radiusCorrect);
			\end{scope}
		\end{scope}
		
		\begin{scope}[shift={(3.5,0)}]
			\node 		(h) at (0,0)	 	{$H_{3}$};
			\begin{scope}[shift={(0,1.3)}]
				\def \n {6}
				\def \radius {0.7cm}
				\def \radiusCorrect {3}
				\def \margin {10} % margin in angles, depends on the radius
				\def \rotate {0}	% to rotate the cycle
				
				\foreach \s in {1,...,\n}
				{
					\node[hhh] at ({360/\n * (\s - 1)+\rotate}:\radius) {};
					\draw[ >=latex]  ({360/\n * (\s - 1)+\margin + \rotate}:\radius)
					arc ({360/\n * (\s - 1)+\margin+\rotate}:{360/\n * (\s)-\margin + \rotate}:\radius);
				}
				\draw ({360/\n * (3 - 1)+\rotate}:\radius-\radiusCorrect)--({360/\n * (0)+\rotate}:\radius-\radiusCorrect)
				({360/\n * (5 - 1)+\rotate}:\radius-\radiusCorrect)--({360/\n * (0)+\rotate}:\radius-\radiusCorrect)
				({360/\n * (2 - 1)+\rotate}:\radius-\radiusCorrect)--({360/\n * (4-1)+\rotate}:\radius-\radiusCorrect)
				({360/\n * (6 - 1)+\rotate}:\radius-\radiusCorrect)--({360/\n * (4-1)+\rotate}:\radius-\radiusCorrect)
				({360/\n * (2 - 1)+\rotate}:\radius-\radiusCorrect)--({360/\n * (6-1)+\rotate}:\radius-\radiusCorrect)
				({360/\n * (3 - 1)+\rotate}:\radius-\radiusCorrect)--({360/\n * (5-1)+\rotate}:\radius-\radiusCorrect);
			\end{scope}
		\end{scope}

		\begin{scope}[shift={(-3.75,-1.5)},scale=0.6]
			\node[hhh] 	(a) at (45:1cm) 	{};
			\node[hhh]  (b) at (135:1cm) 	{};
			\node[hhh] 	(c) at (225:1cm) 	{};
			\node[hhh] 	(d) at (-45:1cm) 	{};

			\node 	(h) at (-90:1.7cm) 	{$H_{4}$};
			
			\draw (a) -- (b)  (c) -- (d) ;
		\end{scope}
		
		\begin{scope}[shift={(-1.25,-1.5)},scale=0.6]
			\node[hhh] 	(a) at (0:1.2cm) 	{};
			\node[hhh]  (b) at (0:0.6cm) 	{};
			\node[hhh] 	(c) at (180:0.6cm) 	{};
			\node[hhh] 	(d) at (180:1.2cm) 	{};
			\node[hhh] 	(e) at (0:0cm) 	{};
			
			\node 	(h) at (-90:1.7cm) 	{$H_{5}$};
			
			\draw (a) -- (b) --(e) -- (c) -- (d);
		\end{scope}

		\begin{scope}[shift={(1.25,-1.5)},scale=0.6]
			\node[hhh] 	(a) at (90:1cm) 	{};
			\node[hhh]  (b) at (90:0.4cm) 	{};
			\node[hhh] 	(c) at (225:1.2cm) 	{};
			\node[hhh] 	(d) at (-45:1.2cm) 	{};
			\node[hhh] 	(e) at (-90:0.3cm) 	{};

			\node 	(h) at (-90:1.7cm) 	{$H_{6}$};
			
			\draw (a) -- (b) --(e) -- (c) -- (d) (d) -- (e);
		\end{scope}

		\begin{scope}[shift={(3.75,-1.5)},scale=0.6]
			\node[hhh] 	(a) at (45:1.2cm) 	{};
			\node[hhh]  (b) at (135:1.2cm) 	{};
			\node[hhh] 	(c) at (225:1.2cm) 	{};
			\node[hhh] 	(d) at (-45:1.2cm) 	{};
			\node[hhh] 	(e) at (180:0cm) 	{};

			\node 	(h) at (-90:1.7cm) 	{$H_{7}$};
			
			\draw (a) -- (b) --(e) -- (c) -- (d) (d) -- (e) -- (a);
		\end{scope}
		
	\end{tikzpicture}	
	\caption{Critically non-split graphs}
	\label{Figure: the graphs in EC(split)}
\end{figure}
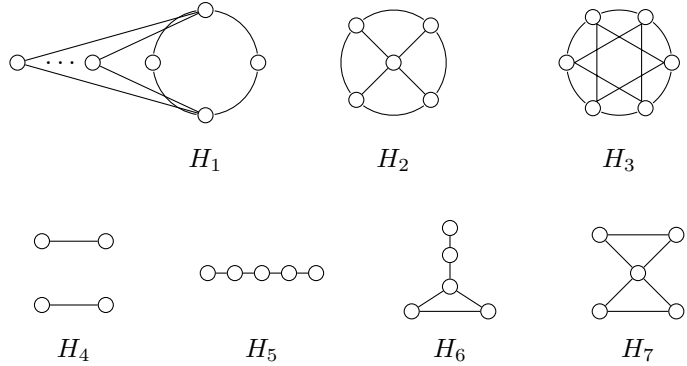

\section{Pseudo-Split Graphs}\label{Section: pseudo-split}
In \cite{blazsik1993graphs}, the family of $\{2K_{2}, C_{4}\}$-free graphs was investigated and later referred to as pseudo-split graphs in \cite{maffray1994linear}. 
Thus, we call a graph that is $\{2K_{2}, C_{4}\}$-exist \emph{non-pseudo-split} graph.

By Theorems \ref{Theorem: characeterization}, \ref{Theorem: 2K2-free characterization} and \ref{Theorem: C4-free characterization} and Propositions \ref{pro: only one free-split for cycles}, \ref{Lemma: Cycle contraction}, and \ref{Lemma: EC(C3)}, we obtain:

\begin{theorem}{\label{Theorem: Pseudo-split characterization short form}}
	Let $G$ be a $C_{5}$-free graph that is non-isomorphic to any graph in Figure \ref{Figure: the graphs in EC(Pseudo-split)}. The graph $G$ is Pseudo-split if and only if any $G$-contraction is Pseudo-split.
\end{theorem}

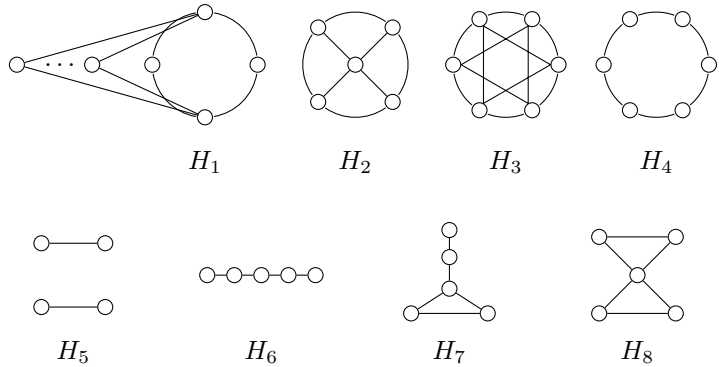
\begin{figure}[ht!]
	\centering
	\begin{tikzpicture}[hhh/.style={draw=black,circle,inner sep=2pt,minimum size=0.2cm}]
		\begin{scope}[shift={(-2,0)}]
			\node 		(h) at (0,0)	 	{$H_{1}$};
			\begin{scope}[shift={(0,1.3)}]
				\def \n {4}
				\def \radius {0.7cm}
				\def \radiusCorrect {3}
				\def \margin {10} % margin in angles, depends on the radius
				\def \rotate {0}	% to rotate the cycle
				
				\node[hhh]  (e) at (-1.5,0)		{};
				\node[hhh]  (f) at (-2.5,0)		{};
				\node[text centered]  (g) at (-1.9,0)		{$\dots$};
				\draw ({360/\n * (2 - 1)+\rotate}:\radius)--(e)--({360/\n * (4 - 1)+\rotate}:\radius)
				({360/\n * (2 - 1)+\rotate}:\radius)--(f)--({360/\n * (4 - 1)+\rotate}:\radius);
				\foreach \s in {1,...,\n}
				{
					\node[hhh,fill=white] at ({360/\n * (\s - 1)+\rotate}:\radius) {};
					\draw[ >=latex]  ({360/\n * (\s - 1)+\margin + \rotate}:\radius)
					arc ({360/\n * (\s - 1)+\margin+\rotate}:{360/\n * (\s)-\margin + \rotate}:\radius);
				}
			\end{scope}
		\end{scope}
		
		\begin{scope}[shift={(0,0)}]
			\node 		(h) at (0,0)	 	{$H_{2}$};
			\begin{scope}[shift={(0,1.3)}]
				\def \n {4}
				\def \radius {0.7cm}
				\def \radiusCorrect {3}
				\def \margin {10} % margin in angles, depends on the radius
				\def \rotate {45}	% to rotate the cycle
				
				\foreach \s in {1,...,\n}
				{
					\node[hhh] at ({360/\n * (\s - 1)+\rotate}:\radius) {};
					\draw[ >=latex]  ({360/\n * (\s - 1)+\margin + \rotate}:\radius)
					arc ({360/\n * (\s - 1)+\margin+\rotate}:{360/\n * (\s)-\margin + \rotate}:\radius);
				}
				\node[hhh]  (e) at (0,0)		{};
				\draw ({360/\n * (1 - 1)+\rotate}:\radius-\radiusCorrect)--(e)--({360/\n * (2 - 1)+\rotate}:\radius-\radiusCorrect)
				({360/\n * (3 - 1)+\rotate}:\radius-\radiusCorrect)--(e)--({360/\n * (4 - 1)+\rotate}:\radius-\radiusCorrect);
			\end{scope}
		\end{scope}
		
		\begin{scope}[shift={(2,0)}]
			\node 		(h) at (0,0)	 	{$H_{3}$};
			\begin{scope}[shift={(0,1.3)}]
				\def \n {6}
				\def \radius {0.7cm}
				\def \radiusCorrect {3}
				\def \margin {10} % margin in angles, depends on the radius
				\def \rotate {0}	% to rotate the cycle
				
				\foreach \s in {1,...,\n}
				{
					\node[hhh] at ({360/\n * (\s - 1)+\rotate}:\radius) {};
					\draw[ >=latex]  ({360/\n * (\s - 1)+\margin + \rotate}:\radius)
					arc ({360/\n * (\s - 1)+\margin+\rotate}:{360/\n * (\s)-\margin + \rotate}:\radius);
				}
				\draw ({360/\n * (3 - 1)+\rotate}:\radius-\radiusCorrect)--({360/\n * (0)+\rotate}:\radius-\radiusCorrect)
				({360/\n * (5 - 1)+\rotate}:\radius-\radiusCorrect)--({360/\n * (0)+\rotate}:\radius-\radiusCorrect)
				({360/\n * (2 - 1)+\rotate}:\radius-\radiusCorrect)--({360/\n * (4-1)+\rotate}:\radius-\radiusCorrect)
				({360/\n * (6 - 1)+\rotate}:\radius-\radiusCorrect)--({360/\n * (4-1)+\rotate}:\radius-\radiusCorrect)
				({360/\n * (2 - 1)+\rotate}:\radius-\radiusCorrect)--({360/\n * (6-1)+\rotate}:\radius-\radiusCorrect)
				({360/\n * (3 - 1)+\rotate}:\radius-\radiusCorrect)--({360/\n * (5-1)+\rotate}:\radius-\radiusCorrect);
			\end{scope}
		\end{scope}

		\begin{scope}[shift={(4,0)}]
			\node 		(h) at (0,0)	 	{$H_{4}$};
			\begin{scope}[shift={(0,1.3)}]
				\def \n {6}
				\def \radius {0.7cm}
				\def \radiusCorrect {3}
				\def \margin {10} % margin in angles, depends on the radius
				\def \rotate {0}	% to rotate the cycle
				
				\foreach \s in {1,...,\n}
				{
					\node[hhh] at ({360/\n * (\s - 1)+\rotate}:\radius) {};
					\draw[ >=latex]  ({360/\n * (\s - 1)+\margin + \rotate}:\radius)
					arc ({360/\n * (\s - 1)+\margin+\rotate}:{360/\n * (\s)-\margin + \rotate}:\radius);
				}
				
			\end{scope}
		\end{scope}

		\begin{scope}[shift={(-3.75,-1.5)},scale=0.6]
			\node[hhh] 	(a) at (45:1cm) 	{};
			\node[hhh]  (b) at (135:1cm) 	{};
			\node[hhh] 	(c) at (225:1cm) 	{};
			\node[hhh] 	(d) at (-45:1cm) 	{};

			\node 	(h) at (-90:1.7cm) 	{$H_{5}$};
			
			\draw (a) -- (b)  (c) -- (d) ;
		\end{scope}
		
		\begin{scope}[shift={(-1.25,-1.5)},scale=0.6]
			\node[hhh] 	(a) at (0:1.2cm) 	{};
			\node[hhh]  (b) at (0:0.6cm) 	{};
			\node[hhh] 	(c) at (180:0.6cm) 	{};
			\node[hhh] 	(d) at (180:1.2cm) 	{};
			\node[hhh] 	(e) at (0:0cm) 	{};
			
			\node 	(h) at (-90:1.7cm) 	{$H_{6}$};
			
			\draw (a) -- (b) --(e) -- (c) -- (d);
		\end{scope}

		\begin{scope}[shift={(1.25,-1.5)},scale=0.6]
			\node[hhh] 	(a) at (90:1cm) 	{};
			\node[hhh]  (b) at (90:0.4cm) 	{};
			\node[hhh] 	(c) at (225:1.2cm) 	{};
			\node[hhh] 	(d) at (-45:1.2cm) 	{};
			\node[hhh] 	(e) at (-90:0.3cm) 	{};

			\node 	(h) at (-90:1.7cm) 	{$H_{7}$};
			
			\draw (a) -- (b) --(e) -- (c) -- (d) (d) -- (e);
		\end{scope}

		\begin{scope}[shift={(3.75,-1.5)},scale=0.6]
			\node[hhh] 	(a) at (45:1.2cm) 	{};
			\node[hhh]  (b) at (135:1.2cm) 	{};
			\node[hhh] 	(c) at (225:1.2cm) 	{};
			\node[hhh] 	(d) at (-45:1.2cm) 	{};
			\node[hhh] 	(e) at (180:0cm) 	{};

			\node 	(h) at (-90:1.7cm) 	{$H_{8}$};
			
			\draw (a) -- (b) --(e) -- (c) -- (d) (d) -- (e) -- (a);
		\end{scope}
		
	\end{tikzpicture}	
	\caption{The critically non-pseudo-split graphs}
	\label{Figure: the graphs in EC(Pseudo-split)}
\end{figure}

\section{Threshold Graphs}
Threshold graphs were characterized as follows:
\begin{theorem}\cite{chvtal1977aggregation}\label{Theorem: threshold graph chvtal charcaterization}
	A given a graph $G$ is threshold if and only if $G$ is $\{2K_{2}$, $P_{4}$, $C_{4}\}$-free.
\end{theorem}
Thus, we call a graph that is $\{2K_{2}, P_{4}, C_{4}\}$-exist \emph{non-threshold} graph.

By Theorems \ref{Theorem: characeterization}, \ref{Theorem: 2K2-free characterization}, \ref{Theorem: P4-free characterization}, and \ref{Theorem: C4-free characterization} and Propositions \ref{pro: only one free-split for cycles}, \ref{Lemma: Cycle contraction}, and \ref{Lemma: EC(C3)}, we obtain:

\begin{theorem}{\label{Theorem: split characterization short form1}}
	Let $G$ be a $C_{5}$-free graph that is non-isomorphic to any graph in Figure \ref{Figure: the graphs in EC(split)}. The graph $G$ is threshold if and only if any $G$-contraction is threshold.
\end{theorem}

\begin{figure}[ht!]
	\centering
	\begin{tikzpicture}[hhh/.style={draw=black,circle,inner sep=2pt,minimum size=0.2cm}]
		\begin{scope}[shift={(-2,0)}]
			\node 		(h) at (0,0)	 	{$H_{1}$};
			\begin{scope}[shift={(0,1.3)}]
				\def \n {4}
				\def \radius {0.7cm}
				\def \radiusCorrect {3}
				\def \margin {10} % margin in angles, depends on the radius
				\def \rotate {0}	% to rotate the cycle
				
				\node[hhh]  (e) at (-1.5,0)		{};
				\node[hhh]  (f) at (-2.5,0)		{};
				\node[text centered]  (g) at (-1.9,0)		{$\dots$};
				\draw ({360/\n * (2 - 1)+\rotate}:\radius)--(e)--({360/\n * (4 - 1)+\rotate}:\radius)
				({360/\n * (2 - 1)+\rotate}:\radius)--(f)--({360/\n * (4 - 1)+\rotate}:\radius);
				\foreach \s in {1,...,\n}
				{
					\node[hhh,fill=white] at ({360/\n * (\s - 1)+\rotate}:\radius) {};
					\draw[ >=latex]  ({360/\n * (\s - 1)+\margin + \rotate}:\radius)
					arc ({360/\n * (\s - 1)+\margin+\rotate}:{360/\n * (\s)-\margin + \rotate}:\radius);
				}
			\end{scope}
		\end{scope}
		
		\begin{scope}[shift={(0.5,0)}]
			\node 		(h) at (0,0)	 	{$H_{2}$};
			\begin{scope}[shift={(0,1.3)}]
				\def \n {4}
				\def \radius {0.7cm}
				\def \radiusCorrect {3}
				\def \margin {10} % margin in angles, depends on the radius
				\def \rotate {45}	% to rotate the cycle
				
				\foreach \s in {1,...,\n}
				{
					\node[hhh] at ({360/\n * (\s - 1)+\rotate}:\radius) {};
					\draw[ >=latex]  ({360/\n * (\s - 1)+\margin + \rotate}:\radius)
					arc ({360/\n * (\s - 1)+\margin+\rotate}:{360/\n * (\s)-\margin + \rotate}:\radius);
				}
				\node[hhh]  (e) at (0,0)		{};
				\draw ({360/\n * (1 - 1)+\rotate}:\radius-\radiusCorrect)--(e)--({360/\n * (2 - 1)+\rotate}:\radius-\radiusCorrect)
				({360/\n * (3 - 1)+\rotate}:\radius-\radiusCorrect)--(e)--({360/\n * (4 - 1)+\rotate}:\radius-\radiusCorrect);
			\end{scope}
		\end{scope}
		
		\begin{scope}[shift={(3.5,0)}]
			\node 		(h) at (0,0)	 	{$H_{3}$};
			\begin{scope}[shift={(0,1.3)}]
				\def \n {6}
				\def \radius {0.7cm}
				\def \radiusCorrect {3}
				\def \margin {10} % margin in angles, depends on the radius
				\def \rotate {0}	% to rotate the cycle
				
				\foreach \s in {1,...,\n}
				{
					\node[hhh] at ({360/\n * (\s - 1)+\rotate}:\radius) {};
					\draw[ >=latex]  ({360/\n * (\s - 1)+\margin + \rotate}:\radius)
					arc ({360/\n * (\s - 1)+\margin+\rotate}:{360/\n * (\s)-\margin + \rotate}:\radius);
				}
				\draw ({360/\n * (3 - 1)+\rotate}:\radius-\radiusCorrect)--({360/\n * (0)+\rotate}:\radius-\radiusCorrect)
				({360/\n * (5 - 1)+\rotate}:\radius-\radiusCorrect)--({360/\n * (0)+\rotate}:\radius-\radiusCorrect)
				({360/\n * (2 - 1)+\rotate}:\radius-\radiusCorrect)--({360/\n * (4-1)+\rotate}:\radius-\radiusCorrect)
				({360/\n * (6 - 1)+\rotate}:\radius-\radiusCorrect)--({360/\n * (4-1)+\rotate}:\radius-\radiusCorrect)
				({360/\n * (2 - 1)+\rotate}:\radius-\radiusCorrect)--({360/\n * (6-1)+\rotate}:\radius-\radiusCorrect)
				({360/\n * (3 - 1)+\rotate}:\radius-\radiusCorrect)--({360/\n * (5-1)+\rotate}:\radius-\radiusCorrect);
			\end{scope}
		\end{scope}

		\begin{scope}[shift={(-3.75,-1.5)},scale=0.6]
			\node[hhh] 	(a) at (45:1cm) 	{};
			\node[hhh]  (b) at (135:1cm) 	{};
			\node[hhh] 	(c) at (225:1cm) 	{};
			\node[hhh] 	(d) at (-45:1cm) 	{};

			\node 	(h) at (-90:1.7cm) 	{$H_{4}$};
			
			\draw (a) -- (b)  (c) -- (d) ;
		\end{scope}
		
		\begin{scope}[shift={(-1.25,-1.5)},scale=0.6]
			\node[hhh] (r) at (-1.5,0) 	{};
			\node[hhh] (s) at (-0.5,0) 	{};
			\node[hhh] (t) at (0.5,0) 	{};
			\node[hhh] (u) at (1.5,0) 	{};
			\node 	(h) at (-90:1.7cm) 	{$H_{5}$};
			
			\draw (r)--(s)--(t)--(u);
		\end{scope}

		\begin{scope}[shift={(1.25,-1.5)},scale=0.6]
			\node[hhh] (r) at (-1.5,-0.8) 	{};
			\node[hhh] (s) at (-0.5,-0.8) 	{};
			\node[hhh] (t) at (0.5,-0.8) 	{};
			\node[hhh] (u) at (1.5,-0.8) 	{};
			\node[hhh] (v) at (0,0.5) 	{};
			\node 	(h) at (-90:1.7cm) 	{$H_{6}$};
			
			\draw (s)--(v)--(r)--(s)--(t)--(u)--(v)--(t);
		\end{scope}

		\begin{scope}[shift={(3.75,-1.5)},scale=0.6]
			\node[hhh] 	(a) at (45:1.2cm) 	{};
			\node[hhh]  (b) at (135:1.2cm) 	{};
			\node[hhh] 	(c) at (225:1.2cm) 	{};
			\node[hhh] 	(d) at (-45:1.2cm) 	{};
			\node[hhh] 	(e) at (180:0cm) 	{};

			\node 	(h) at (-90:1.7cm) 	{$H_{7}$};
			
			\draw (a) -- (b) --(e) -- (c) -- (d) (d) -- (e) -- (a);
		\end{scope}

	\end{tikzpicture}	
	\caption{The critically non-thershold graphs}
	\label{Figure: the graphs in EC(thershold)}
\end{figure}
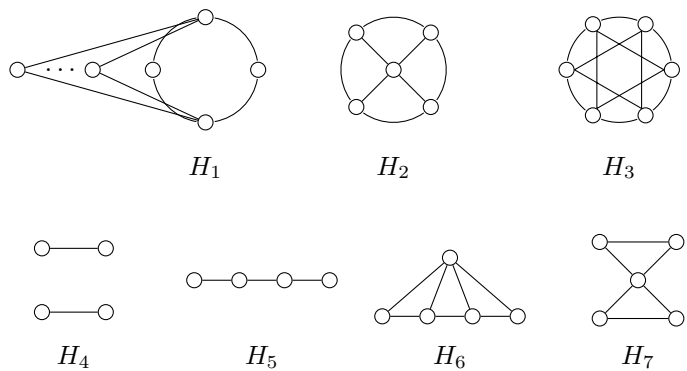

\section{Acknowledgments}
The research presented here is funded by the European Social Fund (ESF).

\bibliographystyle{chicago}
\bibliography{mybib}

\begin{thebibliography}{}

\bibitem[\protect\citeauthoryear{Bertossi}{Bertossi}{1984}]{bertossi1984dominating}
Bertossi, A.~A. (1984).
\newblock Dominating sets for split and bipartite graphs.
\newblock {\em Information Processing Letters\/}~{\em 19\/}(1), 37--40.

\bibitem[\protect\citeauthoryear{Bl{\'a}zsik, Hujter, Pluh{\'a}r, and
  Tuza}{Bl{\'a}zsik et~al.}{1993}]{blazsik1993graphs}
Bl{\'a}zsik, Z., M.~Hujter, A.~Pluh{\'a}r, and Z.~Tuza (1993).
\newblock Graphs with no induced {$C_{4}$} and {$2K_{2}$}.
\newblock {\em Discrete Mathematics\/}~{\em 115\/}(1-3), 51--55.

\bibitem[\protect\citeauthoryear{Bondy and Murty}{Bondy and
  Murty}{2000}]{bondy2000graph}
Bondy, J.~A. and U.~S. Murty (2000).
\newblock {\em Graph Theory}.
\newblock Springer.

\bibitem[\protect\citeauthoryear{Brause, Randerath, Schiermeyer, and
  Vumar}{Brause et~al.}{2019}]{brause2019chromatic}
Brause, C., B.~Randerath, I.~Schiermeyer, and E.~Vumar (2019).
\newblock On the chromatic number of $2{K}_{2}$-free graphs.
\newblock {\em Discrete Applied Mathematics\/}~{\em 253}, 14--24.

\bibitem[\protect\citeauthoryear{Broersma, Patel, and Pyatkin}{Broersma
  et~al.}{2014}]{broersma2014toughness}
Broersma, H., V.~Patel, and A.~Pyatkin (2014).
\newblock On toughness and hamiltonicity of $2{K}_{2}$-free graphs.
\newblock {\em Journal of Graph Theory\/}~{\em 75\/}(3), 244--255.

\bibitem[\protect\citeauthoryear{Cameron and Fitzpatrick}{Cameron and
  Fitzpatrick}{2015}]{cameron2015edge}
Cameron, B. and S.~Fitzpatrick (2015).
\newblock Edge contraction and cop-win critical graphs.
\newblock {\em Australasian Journal of Combinatorics\/}~{\em 63}, 70--87.

\bibitem[\protect\citeauthoryear{Chudnovsky and Seymour}{Chudnovsky and
  Seymour}{2008}]{chudnovsky2008claw}
Chudnovsky, M. and P.~Seymour (2008).
\newblock {C}law-free graphs. {IV}. {D}ecomposition theorem.
\newblock {\em Journal of Combinatorial Theory, Series B\/}~{\em 98\/}(5),
  839--938.

\bibitem[\protect\citeauthoryear{Chudnovsky and Seymour}{Chudnovsky and
  Seymour}{2005}]{chudnovsky2005structure}
Chudnovsky, M. and P.~D. Seymour (2005).
\newblock The structure of claw-free graphs.
\newblock {\em Surveys in {C}ombinatorics\/}~{\em 327}, 153--171.

\bibitem[\protect\citeauthoryear{Chung, Gyárfás, Tuza, and Trotter}{Chung
  et~al.}{1990}]{2k2Free}
Chung, F.~R., A.~Gyárfás, Z.~Tuza, and W.~T. Trotter (1990).
\newblock The maximum number of edges in $2${K}$_{2}$-free graphs of bounded
  degree.
\newblock {\em Discrete Mathematics\/}~{\em 81\/}(2), 129--135.

\bibitem[\protect\citeauthoryear{Chvtal and Hammer}{Chvtal and
  Hammer}{1977}]{chvtal1977aggregation}
Chvtal, V. and P.~Hammer (1977).
\newblock Aggregation of inequalities in integer programming.
\newblock {\em Ann. Discrete Math\/}~{\em 1}, 145--162.

\bibitem[\protect\citeauthoryear{Diner, Paulusma, Picouleau, and Ries}{Diner
  et~al.}{2018}]{diner2018contraction}
Diner, {\"O}.~Y., D.~Paulusma, C.~Picouleau, and B.~Ries (2018).
\newblock Contraction and deletion blockers for perfect graphs and h-free
  graphs.
\newblock {\em Theoretical Computer Science\/}~{\em 746}, 49--72.

\bibitem[\protect\citeauthoryear{El-Zahar and Erd{\H{o}}s}{El-Zahar and
  Erd{\H{o}}s}{1985}]{el1985existence}
El-Zahar, M. and P.~Erd{\H{o}}s (1985).
\newblock On the existence of two non-neighboring subgraphs in a graph.
\newblock {\em Combinatorica\/}~{\em 5\/}(4), 295--300.

\bibitem[\protect\citeauthoryear{Faudree, Flandrin, and
  Ryj{\'a}{\v{c}}ek}{Faudree et~al.}{1997}]{faudree1997claw}
Faudree, R., E.~Flandrin, and Z.~Ryj{\'a}{\v{c}}ek (1997).
\newblock Claw-free graphs a survey.
\newblock {\em Discrete Mathematics\/}~{\em 164\/}(1-3), 87--147.

\bibitem[\protect\citeauthoryear{{Foldes} and {Hammer}}{{Foldes} and
  {Hammer}}{1977}]{hammer1977}
{Foldes}, S. and P.~L. {Hammer} (1977).
\newblock Split graphs.
\newblock {\em {Proc. 8th southeast. Conf. on Combinatorics, Graph Theory, and
  Computing; Baton Rouge 1977.}\/}, 311--315.

\bibitem[\protect\citeauthoryear{Golan and Shan}{Golan and
  Shan}{2016}]{golan2016nonempty}
Golan, G. and S.~Shan (2016).
\newblock Nonempty intersection of longest paths in $2 k\_2 $-free graphs.
\newblock {\em arXiv preprint arXiv:1611.05967\/}.

\bibitem[\protect\citeauthoryear{Hammer and Simeone}{Hammer and
  Simeone}{1981}]{hammer1981splittance}
Hammer, P.~L. and B.~Simeone (1981).
\newblock The splittance of a graph.
\newblock {\em Combinatorica\/}~{\em 1\/}(3), 275--284.

\bibitem[\protect\citeauthoryear{Kratsch, Lehel, and M{\"u}ller}{Kratsch
  et~al.}{1996}]{kratsch1996toughness}
Kratsch, D., J.~Lehel, and H.~M{\"u}ller (1996).
\newblock Toughness, hamiltonicity and split graphs.
\newblock {\em Discrete Mathematics\/}~{\em 150\/}(1), 231--246.

\bibitem[\protect\citeauthoryear{Kriesell}{Kriesell}{2002}]{kriesell2002survey}
Kriesell, M. (2002).
\newblock A survey on contractible edges in graphs of a prescribed vertex
  connectivity.
\newblock {\em Graphs and Combinatorics\/}~{\em 18\/}(1), 1--30.

\bibitem[\protect\citeauthoryear{Maffray and Preissmann}{Maffray and
  Preissmann}{1994}]{maffray1994linear}
Maffray, F. and M.~Preissmann (1994).
\newblock Linear recognition of pseudo-split graphs.
\newblock {\em Discrete Applied Mathematics\/}~{\em 52\/}(3), 307--312.

\bibitem[\protect\citeauthoryear{Meister}{Meister}{2006}]{meister2006two}
Meister, D. (2006).
\newblock Two characterisations of minimal triangulations of $2{K}_{2}$-free
  graphs.
\newblock {\em Discrete Mathematics\/}~{\em 306\/}(24), 3327--3333.

\bibitem[\protect\citeauthoryear{Merris}{Merris}{2003}]{merris2003split}
Merris, R. (2003).
\newblock Split graphs.
\newblock {\em European Journal of Combinatorics\/}~{\em 24\/}(4), 413--430.

\bibitem[\protect\citeauthoryear{Paulusma, Picouleau, and Ries}{Paulusma
  et~al.}{2016}]{paulusma2016reducing}
Paulusma, D., C.~Picouleau, and B.~Ries (2016).
\newblock Reducing the clique and chromatic number via edge contractions and
  vertex deletions.
\newblock In {\em International Symposium on Combinatorial Optimization}, pp.\
  38--49. Springer.

\bibitem[\protect\citeauthoryear{Paulusma, Picouleau, and Ries}{Paulusma
  et~al.}{2019}]{paulusma2019critical}
Paulusma, D., C.~Picouleau, and B.~Ries (2019).
\newblock Critical vertices and edges in {H}-free graphs.
\newblock {\em Discrete Applied Mathematics\/}~{\em 257}, 361--367.

\bibitem[\protect\citeauthoryear{Plummer and Saito}{Plummer and
  Saito}{2014}]{plummer2014note}
Plummer, M.~D. and A.~Saito (2014).
\newblock A note on graphs contraction-critical with respect to independence
  number.
\newblock {\em Discrete Mathematics\/}~{\em 325}, 85--91.

\bibitem[\protect\citeauthoryear{S.~Dhanalakshmi and Manogna}{S.~Dhanalakshmi
  and Manogna}{2016}]{on2k2graphs}
S.~Dhanalakshmi, N.~S. and V.~Manogna (2016).
\newblock On $2{K}_{2}$-free graphs.
\newblock {\em International Journal of Pure and Applied Mathematics\/}~{\em
  109\/}(7), 167--173.

\end{thebibliography}

\end{document}